\documentclass{amsart}
\usepackage{bbm}
\usepackage{enumitem}
\usepackage{subcaption}
\usepackage{tikz}
\usetikzlibrary{matrix}
\tikzstyle{sqvec} = [matrix,thick]
\usepackage[square,numbers]{natbib}
\usepackage{hyperref}
\hypersetup{
    colorlinks=true,
    allcolors=blue,
    pdftitle={Semi-supervised Community Detection via QSDs},
    pdfpagemode=FullScreen,
}

\title{Semisupervised Community Detection via Quasi-Stationary Distributions}
\author{Nicolas Fraiman, Michael Nisenzon}
\date{\today}

\newcommand{\Ex}[1]     {\mathbb{E}\left[#1\right]}
\newcommand{\Var}[1]    {\mathbb{V}\left[#1\right]}
\newcommand{\Prob}[2][] {\mathbb{P}_{#1}\left(#2\right)}
\newcommand{\Ind}[1]    {\mathbbm{1}_{#1}}
\newcommand{\norm}[1]   {\left\| #1 \right\|}

\newcommand{\eps}   {\varepsilon}

\newcommand{\V}    {\mathcal{V}}
\newcommand{\U}    {\mathcal{U}}
\newcommand{\R}    {\mathcal{R}}
\newcommand{\C}    {\mathcal{C}}
\newcommand{\SNR}   {\text{SNR}}

\DeclareMathOperator{\sgn}{Sgn}
\DeclareMathOperator{\pSBM}{PL--SBM}
\DeclareMathOperator{\tr}{tr}

\newtheorem{theorem}{Theorem}
\newtheorem{lemma}{Lemma}
\newtheorem{remark}{Remark}
\theoremstyle{definition}
\newtheorem{definition}{Definition}
\begin{document}

\begin{abstract}
Spectral clustering is a widely used method for community detection in networks. We focus on a semi-supervised community detection scenario in the Partially Labeled Stochastic Block Model (PL-SBM) with two balanced communities, where a fixed portion of labels is known. Our approach leverages random walks in which the revealed nodes in each community act as absorbing states. By analyzing the quasi-stationary distributions associated with these random walks, we construct a classifier that distinguishes the two communities by examining differences in the associated eigenvectors. We establish upper and lower bounds on the error rate for a broad class of quasi-stationary algorithms, encompassing both spectral and voting-based approaches. In particular, we prove that this class of algorithms can achieve the optimal error rate in the connected regime. We further demonstrate empirically that our quasi-stationary approach improves performance on both real-world and simulated datasets.
\end{abstract}

\maketitle

\section{Introduction}
%========================================
Community detection is the problem of identifying subsets of nodes (communities) in a network that are more densely connected internally than to the rest of the network. Such communities often reveal key structural properties in diverse settings, including social networks, biological systems, and information networks. Mathematically, one often models this structure via an unknown partition $\sigma$ of the node set $[n]$. Given the network’s adjacency matrix $A$, the task is to infer this partition $\sigma$, providing insight into the network's organization and underlying processes.

A common generative model for networks with latent community structure is the stochastic block model (SBM) introduced by Holland, Laskey, and Leinhardt~\cite{Holland}. This model has been studied extensively as a theoretical benchmark for evaluating community detection algorithms, for a survey of results see~\cite{Abbe2018}. The balanced stochastic block model (SBM) with two communities generates a random graph on $n$ nodes partitioned into two equally sized groups, $\C_+$ and $\C_-$, where edges appear with probability $p$ within each community and $q$ across communities.

Previous work~\cite{Abbe2016, Hajek16} established that in the connected regime, exact recovery of the two communities is possible for a sufficiently large gap between $p$ and $q$. Spectral algorithms are a popular and easily implemented class of methods for community recovery that leverage the eigenstructure of matrices associated with the graph, such as the adjacency matrix or the Laplacian matrix, to reveal these communities.

In many scenarios, some node labels are partially revealed, making the problem semi-supervised. Incorporating these labels can improve the rate of recovery. In the bounded average degree regime, even an arbitrary small fraction of revealed labels strictly improve the achievable recovery rate, and message passing algorithms that incorporate side information asymptotically achieve the optimal recovery rate~\cite{Mossel2016}. Moreover, partially revealed labels allow recovery for parameters outside the regular recovery region~\cite{Strohmer24}.

Quasi-stationary distributions (QSDs) provide a natural way to integrate partial label information. Given a set of revealed nodes in one community, we define a random walk on the unrevealed nodes while treating the revealed nodes as absorbing states. The limiting distribution of this random walks conditional on not being absorbed is a QSD. See the book by Collet, Martinez and San Martin~\cite{Collet13} for a general reference on the subject. Yaglom~\cite{Yaglom47} introduced the concept and showed that such limiting distributions exist for branching processes conditional on avoiding extinction. Darroch and Seneta~\cite{DS65} characterized QSDs for discrete--time Markov Chains as left principal eigenvectors of a suitably defined submatrix of the transition matrix. Thus, QSDs methods allow eigenvector analysis similar to the classical spectral methods for community detection to bound the error rate. 

Due to the heterogeneous structure of the SBM, we expect an unabsorbed random walk to spend less time in the nodes of the same community as the absorbing set. Darroch and Seneta~\cite{DS65} also interpreted the right eigenvector of the transition matrix as a vector of nonnegative weights on the quasi-stationary distribution such that the entrywise product is the occupation measure (the proportion of time spent at each node) conditional on non-absorption.

In this paper, we apply quasi-stationary distributions to semi-supervised community detection in the partially labeled SBM. By treating revealed nodes as absorbing states and examining the corresponding QSDs, we construct classifiers that leverage both the graph structure and partial label information. We quantify the recovery rate of our methods, and show that our methods extend the range of exact recovery (i.e., identifying all community labels correctly with high probability) in the presence of partially revealed labels all the way to the impossibility threshold shown by Saad and Nosratinia \cite{Nosratinia2018}.

\subsection{Related results}\label{sec:related}
%========================================
For the connected regime, Abb\'e, Bandeira, and Hall~\cite{Abbe2016} and Hajek, Wu, and Xu \cite{Hajek16} established the exact recovery threshold for community detection in the SBM in the connected regime where here $p=a\log n/n$ and $q=b\log n/n$. They proved that below a certain value, exact recovery is information-theoretically impossible and showed that above this value communities can be recovered exactly with spectral methods with refinement, and semidefinite programming respectively.

Zhang and Zhou~\cite{Zhang2016} employed a minimax formulation to determine the optimal error rate for for growing degree networks. Later, Abb\'e, Fan, Wang, and Zhong~\cite{Abbe2020} showed that a simple spectral algorithm using the second eigenvector of the adjacency matrix has a recovery rate that achieves the information-theoretic upper bound for community detection in the connected regime. In particular, the second eigenvector $\bar{\nu}_2$ of the expected adjacency matrix $\Ex{A}$ exactly classifies the communities since $\bar{\nu}_2 \propto \Ind{\C_+} - \Ind{\C_-}$. Unfortunately, the second eigenvector $\nu_2$ of $A$ is not concentrated tightly enough around $\bar{\nu}_2$ for exact recovery. Instead, they used the image $A\bar{\nu}_2/\lambda_2(\bar{A})$ and showed that this product still achieves the optimal recovery rate of $n^{-(\sqrt{a}-\sqrt{b})^2/2}$ with additional lower order terms. Methodologically, they applied the leave--one--out technique to bound the entrywise fluctuations of the $\nu_2$ eigenvector around the image $A\bar{\nu}_2/\lambda(\bar{A})$ to a lower order. As a result, under appropriate choices of the connectivity parameters $p,q$, the second eigenvector $\nu_2$ of $A$ would correctly classify all nodes with high probability, and therefore spectral methods achieve the optimal recovery rate in the connectivity regime.

Subsequently, Deng, Ling, and Strohmer~\cite{Deng2021} extended the results in \cite{Abbe2020} to normalizing under the degree matrix $D$. In particular, for the symmetric degree-normalized graph Laplacian $I-D^{-1/2}AD^{-1/2}$, they used a generalized Davis-Kahan theorem to bound the perturbation in the eigenspace. They found that the Laplacian still gives the same rates for exact recovery in the connected case and showed more stable empirical results.

For the sparse regime, Mossel, Neeman, and Sly \cite{MSN} and separately Massouli\'e \cite{Massoulie} proved the conjecture of Decelle, Krzakala, Moore, and Zdeborov\'a \cite{DKMZ} identifying the so-called ``detectability threshold''. Below this threshold, it becomes information-theoretically impossible to reliably detect communities, regardless of the algorithm used. The detectability threshold is determined by a signal--to--noise ratio (SNR) given by $(a-b)^2/2(a+b)$, where partial recovery is possible when $\SNR > 1$. Moreover, Chin, Rao, and Vu~\cite{Chin15} found that spectral methods for the adjacency matrix achieve near optimal rates of recovery in the sparse case, with bounds on rates of recovery based on the signal--to--noise ratio. 

Community detection with side information has been examined as a potential way to refine and improve the sharp threshold for exact recovery in the connected regime. Saad and Nosratinia~\cite{Nosratinia2018} considered the connected SBM with partially revealed labels subject to noise and found necessary and sufficient conditions for exact recovery, as well as a two-step algorithm to achieve recovery. In this case, the second step is a corrective voting phase.

More recently, Gaudio and Joshi~\cite{Gaudio2024} found an information-theoretic threshold for recovery under partial and noisy side information by comparing to genie-aided estimators where all but one node are revealed. They also found single-step algorithms that achieve exact recovery. In both cases, the parameters for exact recovery do not change substantively. In particular, we cannot improve the range of parameters where exact recovery is possible from the bound in \cite{Abbe2016} unless $1-o(1)$ of the labels are revealed.   

In the bounded degree case, when $p=a/n$ and $q =b/n$, exact recovery is impossible as there are many isolated nodes. However, partial recovery is still possible when $\SNR > 1$. In the presence of side information, Mossel and Xu~\cite{Mossel2016} show that there exist algorithms that achieve optimal performance as well. More recently, Strohmer and Sheng~\cite{Strohmer24} show that in the presence of side information, partial recovery is possible below the $\SNR > 1$ threshold. In particular, any constant fraction $\delta$ of revealed nodes allows for weak recovery for any $\SNR > 0$. 

\subsection{Our contributions}
%========================================
In this paper, we consider the case of a SBM with two balanced communities and partially revealed, noise-free labels. To the best of our knowledge, this is the first paper to consider and apply quasi-stationary distributions to community detection. By treating revealed nodes as absorbing states, we define transition submatrices $P_i$ for each community. Under appropriate connectivity conditions, each quasi-stationary distribution $\mu_i$ exists and assigns lower probabilities to nodes within the same community, facilitating community detection.

We formulate a class of single step estimators for community detection based on the eigenvectors of $P_i$ that includes a simple voting component based on the revealed nodes. The class of estimators is parametrized by a weight to the quasi-stationary component. Using the leave--one--out technique in \cite{Abbe2020} and the generalized Davis-Kahan theorem in \cite{Deng2021} allows us to extend the entrywise eigenvector analysis from the adjacency to the transition matrix. We establish an upper bound on error rates and empirically demonstrate improvements under various parameter settings, particularly in the bounded degree regime.

In addition to an upper bound, we show a minimax lower bound on the error rate with side information over all balanced partitions and partial labellings $(\sigma,\ell)$, analogous to Zhang and Zhou~\cite{Zhang2016}. In particular, the error rate does not change asymptotically with revealed labels in the connected regime and the QSD class achieves the optimal error rate in the connected regime. This aligns with previous work \cite{Gaudio2024,Nosratinia2018}, which confirms the known range for exact recovery under partial information. We present a concise proof via an extension of equivariance as defined in Xu, Jog, and Loh~\cite{Xu2020} under community--preserving permutations of the clustering and revealed labels.

The paper is structured as follows. In Section \ref{sec:related}, we review related work. In Section \ref{sec:main} we define our quasi-stationary algorithm and present upper and lower bounds on its error rate. In Section \ref{sec:proofs}, we show the equivariance lemma needed for the lower bound and develop the matrix and eigenvector concentration results needed for the upper bound. Finally, in Section \ref{sec:experiments}, we empirically compare the QSD method with the classic spectral algorithm over a collection of real and simulated datasets.

\section{Main Results}\label{sec:main}
%========================================
\subsection{Setup and Notation}
%========================================
We first define the model we study, the Partially Labeled Balanced SBM (PL--SBM).

\begin{definition}[Balanced partitions and partial labels]\label{def:pl-sbm}
Let $\V = [n]$, a \emph{partition map} $\sigma: \V \rightarrow \{-1,1\}$ divides the nodes into two corresponding \emph{communities} $\C_+ = \sigma^{-1}(1)$ and $\C_- = \sigma^{-1}(-1)$. Given $\sigma$, \emph{partial labels} are any map $\ell: \V \rightarrow \{-1,0,1\}$ where $\ell(v) = \sigma(v)$ if $\ell(v)\neq 0$. The preimages $\R_+ = \ell^{-1}(1)$, $\R_- = \ell^{-1}(-1)$ are called the \emph{revealed sets} and $\U = \ell^{-1}(0)$ is the \emph{unrevealed set}. We say that $(\sigma, \ell)\in \mathfrak{C}(\delta)$ is in the $\delta$-fraction revealed class of balanced partitions and partial labels if
\[
|\C_+| = |\C_-| = n/2 \text{ and } |\R_+| = |\R_-| = \delta n/2.
\]
\end{definition}

\begin{definition}[Partially Labeled Balanced SBM] 
Given a pair $(\sigma,\ell)\in \mathfrak{C}(\delta)$ and parameters $p$ for the within-community connectivity, and $q$ for the across-community connectivity, a realization of the \emph{Partially Labeled Balanced Stochastic Block Model} $\pSBM(p,q,\sigma,\ell)$ is given by the adjacency matrix $A$ of the graph. The entries of $A$ are symmetric, independent, and given by 
\[
A(u,v) =\begin{cases}
      W(u,v), & \text{if}\ \sigma(u)=\sigma(v)\\
      Z(u,v), & \text{if}\ \sigma(u)\neq\sigma(v)
    \end{cases}
\]
where $W(u,v)$ is a Bernoulli($p$) variable, $Z(u,v)$ is a Bernoulli($q$) variable.
\end{definition}

Recall that unrevealed nodes be denoted by $\U$, let the unrevealed nodes in community $\C_i$ be defined by $\U_i = \C_i \setminus \R_i$. Applying a permutation to make the respective subgroups contiguous gives us Figure \ref{fig:labels} below.

\begin{figure}[ht]
\centering
\begin{tikzpicture}[scale=.7, transform shape]
    \draw[thick] (0, 0) rectangle (10, 1);
    \draw[thick, dashed] (5, 0) -- (5, 1);
    \node[below] at (2.5, 0) {\huge $\C_+$};
    \node[below] at (7.5, 0) {\huge $\C_-$};
    \draw[thick, red] (0, 0) rectangle (2, 1);
    \node[above] at (1, 0) {\color{red} \huge $\R_+$};
    \draw[thick, blue] (8, 0) rectangle (10, 1);
    \node[above] at (9, 0) {\color{blue} \huge $\R_-$};
    \node[above] at (3.5, 0) {\huge $\U_+$};
    \node[above] at (6.5, 0) {\huge $\U_-$};
\end{tikzpicture}
\caption{\label{fig:labels} Revealed and unrevealed nodes up to permutation.}
\end{figure}
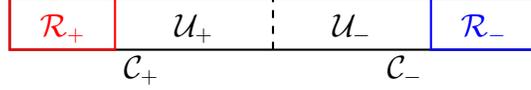

In addition to the adjacency matrix, we also consider the transition matrix $P = D^{-1}A$ which results from normalizing by degree. For our procedure, we consider only revealed nodes in one community $\R_i$ at a time. For each community $i$, we 
consider the restriction to the remaining nodes which we indicate by $\V_i = \V \setminus \R_i$. The adjacency and transition submatrices are defined as follows.

\begin{definition}[Transition submatrices and eigenvectors] 
For each community $i$, we define the submatrices
\[
A_i = A|_{\V_i,\V_i} \quad\text{and}\quad P_i = P|_{\V_i,\V_i}.
\]
We denote the left and right principal eigenvectors of $P_i$ by $\mu_i$ and $\pi_{i}$ respectively. We normalize the left eigenvectors $\mu_i$ such that $\norm{\mu_i}_1 = 1$ in this paper to preserve the intuition of a probability distribution and the right eigenvectors $\pi_i$ such that $\norm{\pi_i}_2 = 1$ so we can apply the Davis-Kahan theorem.
\end{definition}

Up to permutation, the resulting submatrices $P_i$ are shown in Figure \ref{fig:submatrices} below.

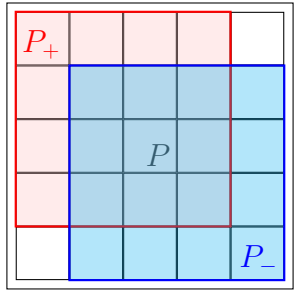
\begin{figure}[ht]
\centering
\begin{tikzpicture}[scale=.7, transform shape]
    \matrix[matrix of nodes, nodes in empty cells,minimum width=.7cm, minimum height=.7cm, draw=black] (P) {
        |[draw=black]| & |[draw=black]| & |[draw=black]| & |[draw=black]| & |[draw=black]| \\
        |[draw=black]| & |[draw=black]| & |[draw=black]| & |[draw=black]| & |[draw=black]| \\
        |[draw=black]| & |[draw=black]| & |[draw=black]| & |[draw=black]| & |[draw=black]| \\
        |[draw=black]| & |[draw=black]| & |[draw=black]| & |[draw=black]| & |[draw=black]| \\
        |[draw=black]| & |[draw=black]| & |[draw=black]| & |[draw=black]| & |[draw=black]| \\
    };
    
    \node[above left] at (P-3-4.south west) {\huge $P$};

    \draw[thick, red, fill= pink, fill opacity = .3] (P-1-1.north west) rectangle (P-4-4.south east);
     \draw[thick, blue, fill= cyan, fill opacity = .3] (P-2-2.north west) rectangle (P-5-5.south east);
    
    \node[above right, text=red] at (P-1-1.south west) {\huge $P_+$};
    \node[above left, text=blue] at (P-5-5.south east) {\huge $P_-$};
\end{tikzpicture}
\caption{\label{fig:submatrices} Submatrices of the transition matrix up to permutation.}
\end{figure}
The above submatrix definitions extend naturally to the expected adjacency matrix $\bar{A} = \Ex{A}$, expected degree matrix $\bar{D} = \Ex{D}$ and expected transition matrix $\bar{P} = \bar{D}^{-1}\bar{A}$. All matrices are interpreted as functions, with $A(u,v) = A_{u,v}$ for $u,v \in \V$. 

Throughout the paper, for concision, we use $\lambda_1,\lambda_2$ alone to refer to the first and second eigenvalues $\lambda_1 = \lambda_1(\bar{P}_+) = \lambda_1(\bar{P}_-)$ and $\lambda_2 =\lambda_2(\bar{P}_+) = \lambda_2(\bar{P}_-)$ and we denote the $i$th eigenvalue of a matrix $M$ by $\lambda_i(M)$. In addition we denote the degree of node $u$ by $d(u)$ and the the expected degree of each node by $\bar{d} = (a+b)\log n/2$.

We proceed to describe the eigenvector of the expected transition matrix $\bar{P}_i$ and establish constants used throughout the paper. We note that the expected adjacency matrix $\bar{A}$ follows the form
\[
\bar{A} = \frac{\log n}{n} \begin{bmatrix}a & b \\ b & a\end{bmatrix} \otimes \mathbf{J}_{n/2\times n/2}
\]
where $\mathbf{J}$ is the all-ones matrix.

Then $\bar{P}_i$ is also a rank-two block matrix with two distinct values and the principal eigenvector $\bar{\pi}_i$ is proportional to $\Ind{\U_i} + \rho\cdot \Ind{\C_{-i}}$ for some $\rho > 0$. Then we see that 
\[
(1-\delta)\frac{na}{2}+\rho \frac{nb}{2} = \frac{1}{\rho}\left((1-\delta)\frac{nb}{2}+\delta \frac{na}{2}\right).
\]
Rearranging terms gives us the quadratic equation
\[
\rho^2 - \frac{a\delta \rho}{b} - (1-\delta) =0
\]
which has a positive solution of
\[
\rho = \frac{1}{2b}(a\delta + \sqrt{a^2\delta^2+4(1-\delta)b^2}).
\]

We use $\rho$ to normalize the mean--field eigenvector $\bar{\pi}_i$ such that $\norm{\bar{\pi}_i}_2 = 1$.

\begin{definition}[Mean--field eigenvector]\label{def:mf-eig} For $i \in \{+,-\}$, we normalize the eigenvector $\bar{\pi}_i$ as
    \[\bar{\pi}_i = \frac{1}{\gamma\sqrt{n}}(\Ind{\U_i} + \rho\cdot \Ind{\C_{-i}})\]
    where $\rho = \dfrac{1}{2b}(a\delta + \sqrt{a^2\delta^2+4(1-\delta)b^2})$ and $\gamma =\sqrt{\dfrac{1-\delta+\rho^2}{2}}$.
\end{definition}

Finally, we define the error rate for a recovery algorithm. Any recovery algorithm $\hat{\sigma}$ takes in an adjacency matrix $A$ and partial labelling $\ell$ and returns an estimator $\hat{\sigma}[A,\ell]: \U \to \{1,-1\}$ that can be compared to the true assignment $\sigma$ using the Hamming distance to find the error rate. In the unlabeled case, this holds up to a global flip between the communities~\cite{Abbe2020}, but partial labellings allow us to identify the respective communities so such a realignment is no longer necessary. Then for a fixed clustering $\sigma$ and a fixed partial labelling $\ell$, the error rate for $\hat{\sigma}$ is the normalized Hamming distance on the unrevealed nodes $\U$.

\begin{definition}[Error rate] For an estimator $\hat{\sigma}$, the error rate $r(\hat{\sigma})$ is defined as
\[
r(\hat{\sigma}) = \frac{1}{|\U|}\sum_{v\in \U} \Ind{ \hat{\sigma}(v) \neq \sigma(v)}.
\]    
\end{definition}

\subsection{Minimax bound for semi-supervised clustering}
%========================================
Following the approach of Zhang and Zhou~\cite{Zhang2016}, we derive a minimax bound that suggests that the quasi-stationary method in the connected regime has asymptotic performance matching the optimal unsupervised case. This requires showing a partially revealed equivariance lemma analogous to their global-to-local lemma.

Throughout this paper, we only consider permutation equivariant and community equivariant estimators $\hat{\sigma}$ extending the definitions of Xu, Jog, and Loh~\cite{Xu2020}.

Let $\pi \in \Pi_{n}$ be a permutation on all $n$ nodes and $A^\pi, \sigma^\pi, \ell^\pi$ be the permuted adjacency matrix, labels, and revealed labels, where $A^\pi(\pi(u),\pi(v)) = A(u,v)$.

\begin{definition}[Equivariance]
$\hat{\sigma}$ is permutation equivariant if permuting the labels does not affect the estimate of the permuted nodes
\[
\hat{\sigma}[A,\ell] = \hat{\sigma}[A^\pi,\ell^\pi]\circ\pi
\] 
In addition, $\hat{\sigma}$ is community equivariant if permuting the communities of the revealed nodes commutes with the estimator
\[
\hat{\sigma}[A, -\ell] = -\hat{\sigma}[A,\ell].
\]
\end{definition}
We note that permutation and community equivariance both hold for a wide variety of estimators, including the quasi-stationary method used in this paper.

In our setting, the minimax theorem corresponds to the following theorem.
\begin{theorem}\label{th:minimax} 
For equivariant estimators $\hat{\sigma}$ we have
 \[
\inf_{\hat{\sigma}} \Ex{r(\hat{\sigma})} \geq n^{-(1+o(1))(\sqrt{a}-\sqrt{b})^2/2}.
\]
\end{theorem}

\begin{proof}
Let $\hat{\sigma}$ be a permutation and community equivariant estimator that satisfies
\[
\Ex{r(\hat{\sigma})} = \inf_{\hat{\sigma}} \Ex{r(\hat{\sigma})}. 
\]
By Lemma \ref{equivariance}, for any nodes $u,v \in \U$  
\[
\Prob{\hat{\sigma}(u) \neq \sigma(u)} = \Prob{\hat{\sigma}(v) \neq \sigma(v)}.
\]

We then apply linearity to see that
\[
\Ex{r(\hat{\sigma})} 
= \frac{1}{|\U|} \sum_{v\in \U} \Ex{\Ind{\hat{\sigma}(v) \neq \sigma(v)}}
=\frac{1}{|\U|}\sum_{v\in \U}  \Prob{\hat{\sigma}(v) \neq \sigma(v)} 
= \Prob{\hat{\sigma}(v) \neq \sigma(v)} 
\]
for any unrevealed node $v \in \U$.

Lemma 5.1 in Zhang and Zhou~\cite{Zhang2016} provides a lower bound on the entrywise error rate by considering the case where all other labels $[n] \setminus v$ are revealed and then taking a likelihood test between $v \in \C_+$ and $v \in \C_-$. In particular, the entrywise error rate is minimized by taking a majority vote on all $[n] \setminus v$. Since full information includes any partial information, this must also be a lower bound for the entrywise error rate of the PL--SBM. Therefore, 
\[
\Prob{\hat{\sigma}(v) \neq \sigma(v)} \geq \eps \Prob{\sum_{i=1}^{n/2} (W_i-Z_i) \leq 0}
\]
for some $\eps > 0$, where the $W_i$ and $Z_i$ are i.i.d. Bernoulli$(a\log n/n)$ and \\ Bernoulli$(b\log n/n)$ respectively. By Lemma 5.2 in~\cite{Zhang2016}
\[
\Prob{\sum_{i=1}^{n/2} (W_i-Z_i) \leq 0} \geq n^{-(1+o(1))(\sqrt{a}-\sqrt{b})^2/2}.
\]

We conclude that 
\[
\inf_{\hat{\sigma}} \Ex{r(\hat{\sigma})} \geq n^{-(1+o(1))(\sqrt{a}-\sqrt{b})^2/2}. \qedhere
\]
\end{proof}
\subsection{Performance of QSD method}
%========================================
The existence of revealed nodes $\R_i$ allows us to define  voting methods, in particular quasi-stationary distribution (QSD) methods and simple voting. We define the quasi-stationary method over two communities by taking the difference between the corresponding principal eigenvectors.

\begin{definition} (QSD score) \label{def:qsd-def}
For each node $u \in \U$, the quasi-stationary score $Q(u)$ is defined as the difference between the left eigenvectors
\[
Q(u) = \mu_-(u) - \mu_+(u).
\]
with a corresponding estimator
\[
\hat{\sigma}_Q(u) = \sgn(Q(u))
\]
\end{definition}

We define simple voting by subtracting the revealed group connections.

\begin{definition} (Simple voting score)
For each $u \in \U$, simple voting $S(u)$ is given by
\[
S(u) = \sum_{v \in \R_+} A(u,v) - \sum_{v \in \R_-} A(u,v)
\]
with a corresponding estimator
\[
\hat{\sigma}_S(u) = \sgn(S(u))
\]
\end{definition}

For a fixed choice of the connectivity parameters $a,b$ and an estimator $\hat{\sigma}$, we would like to bound the expected error rate $\Ex{r(\hat{\sigma})}$ over all networks generated from $\sigma,\ell$. Through equivariance as applied in Theorem \ref{equivariance}, we relate the global error rate to an entrywise error rate $\Prob{\hat{\sigma}(u) \neq \sigma(u)}$ for any unseen node $u \in \U$. The entrywise error is in turn bounded using Lemma \ref{lem:mean-field-error} in terms of an exponential rate $I$ analogous to the information-theoretic optimum in the unsupervised case \cite{Zhang2016}. A natural parametrization of the rate $I$ is by weights $(\alpha,\beta)$ on some functions of the unrevealed and revealed nodes, respectively. As a result, for some constants $(\alpha,\beta)$, the desired bounds on the error rate follow the form
\[
\Ex{r(\hat{\sigma})} \leq n^{-I(\alpha,\beta)}.
\]

We state the main results of our paper. We first show a bound on the quasi-stationary estimator $\hat{\sigma}_Q$.

\begin{theorem}\label{th:default-rate}  
Given $(\sigma,\ell)\in\mathfrak{C}(\delta)$ consider a $\pSBM(p,q,\sigma,\ell)$ with $p= a\log n/n$, $q= b\log n/n$ and $a+b > 4$. Then the quasi-stationary estimator $\hat{\sigma}_Q$ has an error rate of at most
 \[
\Ex{r(\hat{\sigma}_{Q})} \leq n^{-(1+o(1))I(\rho-1,\rho)}+o(n^{-1}).
\]
for 
\[
I(\alpha,\beta) = \frac{1}{2}\sup_{\theta > 0} a+b - (1-\delta)ae^{-\theta \alpha} -(1-\delta)be^{\theta \alpha} -\delta ae^{-\theta \beta} - \delta be^{\theta \beta}.
\]
where $\rho$ is defined in Definition \ref{def:mf-eig}.
\end{theorem}
\begin{proof}
By Lemma \ref{equivariance}, we see that the expected error rate is equal to the probability of misclassifying any node $u$
\[
\Ex{r(\hat{\sigma}_Q)} 
= \frac{1}{|\U|}\sum_{v\in \U} \Prob{\hat{\sigma}_Q(v) \neq \sigma(v)}
= \Prob{\hat{\sigma}_Q(u) \neq \sigma(u)} = \Prob{\sigma(u)Q(u) < 0}.
\]

By Lemma \ref{lem:left-eig-difference}, we can approximate $\mu_--\mu_+$ by a difference of right eigenvectors $D_-\pi_- - D_+\pi_+$ with a scaling factor of $\frac{\lambda_1}{\norm{\bar{A}_i\bar{\pi}_i}_1} = \frac{\lambda_1}{\bar{c}\sqrt{n}\log n}$ for some constant $\bar{c}$ with value explicitly stated in Definition \ref{def:mixed-def}. In particular, with probability $1-o(n^{-1})$
\begin{align}
 \max_{u \in \U} \left|(\mu_--\mu_+)(u) - \frac{\lambda_1}{\bar{c}\sqrt{n}\log n}d(u)(\pi_--\pi_+)(u)\right| = o(n^{-1}). \label{eq:left-proof}
\end{align}

As a result, the error from approximating $Q(u)$ by right eigenvectors $\pi_i$ is controlled by upper-bounding with an expression of order $O(n^{-1})$. It is then sufficient to show that there is a $\eps(a,b) > 0$ such that
\begin{align}
& \Prob{\sigma(u)Q(u) < 0} \notag\\
  &\qquad = \Prob{\sigma(u)(\mu_- - \mu_+)(u) < 0} \notag\\
  &\qquad \leq \Prob{\sigma(u)\left((\mu_- - \mu_+)(u) - \frac{\lambda_1}{\bar{c}\sqrt{n}\log n}d(u)(\pi_--\pi_+)(u) \right) \leq -\frac{\eps}{n}} \label{eq:main1}\\
 &\qquad \quad + \Prob{\frac{\lambda_1}{\bar{c}\sqrt{n}\log n}\sigma(u)d(u)(\pi_--\pi_+)(u) \leq \frac{\eps}{n}} \label{eq:main2} 
\end{align}
We bound both terms separately. First, by our work above in \eqref{eq:left-proof}, \eqref{eq:main1} is bounded by 
\[
\Prob{\sigma(u)\left((\mu_- - \mu_+)(u) - \frac{\lambda_1}{\bar{c}\sqrt{n}\log n}d(u)(\pi_--\pi_+)(u) \right) \leq -\frac{\eps}{n}} = o(n^{-1})
\]
for any constant $\eps > 0$.

We bound $\eqref{eq:main2}$ by applying the right eigenvector bound in Lemma \ref{lem:right-eig-error}. We scale by the constant factor $\gamma = \sqrt{(1-\delta +\rho^2)/2}$ defined in Definition \ref{def:mf-eig}. We then note that $\bar{c}$ is also a constant and substitute $\eps_1 = \eps\bar{c}\gamma$ to get
\begin{align*}
&\Prob{\frac{\lambda_1}{\bar{c}\sqrt{n}\log n}\sigma(u)d(u)(\pi_--\pi_+)(u) \leq \frac{\eps}{n}} \\
& \qquad \qquad = \Prob{\sigma(u)\gamma\lambda_1\sqrt{n}d(u)(\pi_--\pi_+)(u) \leq \eps\bar{c}\gamma\log n} \\
&\qquad \qquad =\Prob{\sigma(u)\gamma\lambda_1\sqrt{n}d(u)(\pi_--\pi_+)(u)\leq \eps_1 \log n} \\
&\qquad \qquad \leq n^{-I(\rho-1,\rho)+O(\eps_1)} +o(n^{-1})
\end{align*}
where we use the definition of $I$ from Lemma \ref{lem:right-eig-error} to get
\[
I(\rho-1,\rho) = \frac{1}{2}\sup_{\theta > 0} (a+b) - (1-\delta)ae^{-\theta(\rho-1)} -(1-\delta)be^{\theta(\rho-1)} -\delta ae^{-\theta\rho} - \delta be^{\theta\rho}.
\]
We note that $I(\rho-1,\rho)$ does not have a closed--form solution.

Then we combine our bounds on \eqref{eq:main1} and \eqref{eq:main2} to conclude that
\[
\Prob{\sigma(u)(\mu_--\mu_+)(u) < 0} \leq n^{-I(\rho-1,\rho)+O(\eps_1)} + o(n^{-1}).
\]
The result follows by taking $\eps_1 \rightarrow 0$.
\qedhere
\end{proof}

In addition to the quasi-stationary estimator, we also analyze the simple voting estimator.

\begin{theorem}\label{th:simple-voting}  
Under the conditions of Theorem \ref{th:default-rate}, the simple voting estimator $\hat{\sigma}_S$ achieves the following error rate
 \[
\Ex{r(\hat{\sigma}_{S})} \leq n^{-(1+o(1))\delta(\sqrt{a}-\sqrt{b})^2/2}.
\]
\end{theorem}
\begin{proof}
We first note that the simple voting estimator satisfies
\[
\hat{\sigma}_{S}(u) = \sgn(S(u)).
\]
As in Theorem \ref{th:default-rate}, we apply equivariance via Lemma \ref{equivariance} to show that the expected error rate is equal to the probability of misclassifying any node $u$
\[
\Ex{r(\hat{\sigma}_S)} = \Prob{\sigma(u)\hat{\sigma}_S(u) < 0} = \Prob{\sigma(u)S(u) < 0}.
\]
Note that $\sigma(u)S(u) = \sigma(u)\left(\sum_{v \in \R_+} A(u,v) - \sum_{v \in \R_-} A(u,v)\right)$ is equal in distribution to $\sum_{i=1}^{\delta n/2}(W_i - Z_i)$ where $W_i$ and $Z_i$ are i.i.d.~Bernoulli$(a\log n/n)$ and Bernoulli$(b\log n/n)$ respectively.
Then we apply the Chernoff bound in Lemma \ref{lem:total-chernoff1} with $\alpha = 0,\beta =1$ to conclude that there exists a $\eps > 0$ such that
\[
\Prob{\sigma(u)S(u) < 0} \leq \Prob{\sum_{i=1}^{\delta n/2}W_i - Z_i \leq \eps \log n} \leq n^{-I(0,1)+O(\eps)}
\]
where
\[
I(0,1)  = \frac{\delta}{2}\sup_{\theta > 0} (a+b) - ae^{-\theta} - be^{\theta}
\]
We set $\theta = \log(a/b)/2$ to get that $I(0,1)  \geq \delta(\sqrt{a}-\sqrt{b})^2/2$, which yields the desired bound
\[
\Ex{r(\hat{\sigma}_S)} = n^{-\delta(\sqrt{a}-\sqrt{b})^2/2 + O(\eps)}.
\]
The result follows by taking $\eps \rightarrow 0$.
\qedhere
\end{proof}

\begin{remark}
%========================================
We note that simple voting achieves exact recovery for sufficiently large values of $\delta$. Saad and Nosratinia~\cite{Nosratinia2018} previously showed that when $\log(1-\delta) = -c\log n$ for $c \in (0,1)$ and $\sqrt{a}-\sqrt{b} > \sqrt{2-2c}$, applying a union bound over $(1-\delta)n$ nodes we get an error rate of 
\[
\Ex{r(\hat{\sigma}_s)} \leq n^{1-c}\cdot \left(n^{-(1+o(1))(\sqrt{a}-\sqrt{b})^2/2}\right) = o(1).
\]
In this case, spectral methods are unnecessary as simple voting achieves exact recovery and voting has an $O(n)$ runtime advantage over spectral methods. This is also shown in more generality in Gaudio and Joshi~\cite{Gaudio2024}, with the threshold being extended to the case of unbalanced communities with distinct connectivity parameters under noisy labels.
\end{remark}

As neither $\hat{\sigma}_Q$ nor $\hat{\sigma}_S$ achieves the information-theoretically optimal recovery rate of $n^{-(1+o(1))(\sqrt{a}-\sqrt{b})^2/2}$ that some unsupervised methods do \cite{Abbe2020}, we also consider a weighted sum of the previous estimators that does achieve the optimal rate.

\begin{definition} (Mixed QSD method) \label{def:mixed-def} For each node $u \in \U$, the mixed QSD method is defined as
\[
M(u) = \bar{c}\gamma n \log n Q(u) -  S(u)
\]
with a corresponding estimator
\[
\hat{\sigma}_{M}(u) = \sgn(M(u))
\]
where
\begin{align*}
    &\rho = \frac{1}{2b}(a\delta + \sqrt{a^2\delta^2+4(1-\delta)b^2}) \\
    &\gamma = \sqrt{\frac{(1-\delta+\rho^2)}{2}} \\ 
    &\bar{c} = \frac{\norm{\bar{A}_i\bar{\pi}_i}_1}{\sqrt{n}\log n} = \frac{1}{4}\left(a(\rho + (1-\delta)^2) + b (\rho + 1)(1-\delta)\right)
\end{align*}
\end{definition}

\begin{theorem}\label{th:main-result}  
Under the conditions of Theorem \ref{th:default-rate}, the mixed estimator $\hat{\sigma}_M$ achieves an error rate of
\[
\Ex{r\left(\hat{\sigma}_{M}\right)} \leq n^{-(1+o(1))(\sqrt{a}-\sqrt{b})^2/2}+o(n^{-1}).
\]
In particular, the mixed QSD method achieves the optimal asymptotic error rate.
\end{theorem}
\begin{proof} 
The proof proceeds like in Theorem \ref{th:default-rate}. By equivariance, we reduce the expected error rate to the entrywise error rate for any unseen $u \in \U$
\[
\Ex{r\left(\hat{\sigma}_{M}\right)} = \Prob{\sigma(u)M(u) < 0}
\]
where 
\[
M(u)= \gamma\bar{c}n\log n(\mu_- - \mu_+)(u) - S(u).
\]
Then we expand out the QSD component $(\mu_- - \mu_+)(u)$ and approximate by right eigenvectors to get
\begin{align}
& \Prob{\sigma(u)M(u) < 0} \notag\\
  &\quad =  \Prob{\sigma(u)\left(\gamma\bar{c}n\log n(\mu_- - \mu_+)(u) - S(u)\right) < 0} \notag\\
  &\quad \leq \Prob{\sigma(u)\gamma\bar{c}n\log n\left((\mu_- - \mu_+)(u) - \frac{\lambda_1}{\bar{c}\sqrt{n}\log n}d(u)(\pi_--\pi_+)(u) \right) \leq -\eps \log n} \label{eq:mixed1}\\
 & \quad \quad + \Prob{\sigma(u)\left(\gamma\lambda_1\sqrt{n}d(u)(\pi_--\pi_+)(u) - S(u)\right)\leq \eps \log n} \label{eq:mixed2} 
\end{align}
We first bound the error of the right eigenvector approximation \eqref{eq:mixed1} by $o(n^{-1})$ exactly as in Theorem \ref{th:default-rate}. Next, we bound \eqref{eq:mixed2} with the right eigenvector bound in Lemma \ref{lem:right-eig-error} to see that for some $\eps > 0$
\[
\Prob{\sigma(u)\left(\gamma \lambda_1 \sqrt{n}d(u)(\pi_--\pi_+)(u) - S(u)\right) \leq \eps \log n} \leq n^{-I(\rho-1,\rho -1)+O(\eps)} + o(n^{-1})
\]
where
\[
   I(\rho-1,\rho-1) = \frac{1}{2} \left(\sup_{\theta > 0} a+b - ae^{-\theta(\rho -1)}-be^{\theta(\rho -1)}\right)
\]
This optimization problem is solved in Theorem \ref{th:simple-voting}, leading to a bound of 
\[
\Prob{\sigma(u)\left(\gamma \lambda_1 \sqrt{n}d(u)(\pi_--\pi_+)(u) - S(u)\right) \leq \eps \log n}  \leq n^{-(\sqrt{a}-\sqrt{b})^2/2+O(\eps)} + o(n^{-1})
\]
and an expected error rate of
\[
\Ex{r\left(\hat{\sigma}_{M}\right)} \leq n^{-(\sqrt{a}-\sqrt{b})^2/2+O(\eps)}+o(n^{-1}).
\]
The result follows by taking $\eps \rightarrow 0$.
\qedhere
\end{proof}

\section{Proofs}\label{sec:proofs}
%========================================
\subsection{Equivariance}
%========================================

In this section, we prove the equivariance lemma, showing that the probability of misclassifying any label is equal over all permutation and community equivariant estimators $\hat{\sigma}$.

\begin{lemma}\label{equivariance} 
For any unrevealed nodes $u,v \in \U$ and an equivariant estimator $\hat{\sigma}$
\[
\Prob{\hat{\sigma}(u) \neq \sigma(u)} = \Prob{\hat{\sigma}(v) \neq \sigma(v)}. 
\]
\end{lemma}

\begin{proof}
Let $\hat{\sigma}$ be a permutation and community equivariant estimator. We show that 
\[
\Prob{\hat{\sigma}(u) \neq \sigma(u)} = \Prob{\hat{\sigma}(v) \neq \sigma(v)}. 
\]
for $u,v \in \U$. We have two cases: either $u,v$ are in the same community and $\sigma(u) = \sigma(v)$ or $u,v$ are in different communities and $\sigma(u) \neq \sigma(v)$. For both cases, we show equality via a community-preserving permutation $\pi$ on $\U$ such that if $u,v$ are in the same community, then $\pi(u),\pi(v)$ remain in the same community, and $\pi(u) = v$, $\pi(v) = u$.

In the first case, where $\sigma(u) = \sigma(v)$, let $\pi \in \Pi_{|\U|}$ be the transposition $(u,v)$ that leaves all other nodes unchanged. In the case where $\sigma(u) \neq \sigma(v)$, let $\pi$ be the permutation that maps nodes in $\U_+$ to $\U_-$ and transposes $u,v$. Clearly both choices preserve communities and transpose $u,v$. In addition, since the underlying SBM is balanced with balanced label cardinalities, $A^\pi$ is identically distributed to $A$ in both cases. We also note that there exists $\tau \in \{-1,1\}$ such that $\ell^\pi = \tau\cdot \ell$ for both cases.

Applying permutation equivariance, community equivariance, and exchangeability we show the desired equality for both cases as follows
\begin{align*}
&\Prob{\hat{\sigma}[A,\ell](u) \neq \sigma(u)} \\
& \qquad \qquad = \Prob{\hat{\sigma}[A^\pi, \ell^\pi](\pi(u)) \neq \sigma(u)} \\
& \qquad \qquad= \Prob{\hat{\sigma}[A^\pi, \tau \cdot \ell](v) \neq \tau \cdot \sigma(v)} \\
& \qquad \qquad = \Prob{\hat{\sigma}[A, \tau \cdot \ell](v) \neq \tau \cdot \sigma(v)} \\
& \qquad \qquad= \Prob{\tau \cdot \hat{\sigma}[A,\ell](v) \neq \tau \cdot \sigma(v)} \\
& \qquad \qquad = \Prob{\hat{\sigma}[A,\ell](v) \neq \sigma(v)} 
\end{align*}

where the first equality follows by permutation equivariance, the second equality follows by choice of $\tau \in \{-1,1\}$, the third equality follows by exchangeability under community preserving transformations, and the fourth equality by community equivariance.

Therefore, for any $u,v \in \U$
\[
\Prob{\hat{\sigma}(u) \neq \sigma(u)} = \Prob{\hat{\sigma}(v) \neq \sigma(v)}. \qedhere
\]
\end{proof}

%========================================
\subsection{Normalized matrix bounds}
%========================================
We start by noting the following concentration result of Lei and Rinaldo~\cite{Lei2015} on the adjacency matrix.

\begin{theorem}[Lei and Rinaldo~\cite{Lei2015} Theorem 5.2]\label{th:lei-rinaldo}  
Let $\max_{u,v} \bar{A}(u,v) \geq k\log n/n$. For any $r > 0$, there exist $c_1,c_2 >0$ such that $\norm{A- \bar{A}}_2 < \max_{u,v} c_1(r,k)\sqrt{n\bar{A}(u,v)}$ with probability $1-c_2n^{-r}$.
\end{theorem}

In particular, it is well-known~\cite{Abbe2020}  that the result holds for $r = 3$. For the expected adjacency matrix $\bar{A}$, the nodes have degrees of order $O(\log n)$. To extend Lei and Rinaldo to the normalized matrix $P$, we must show that the minimum degree is also of order $O(\log n)$ with high probability. 

\begin{lemma}\label{lem:min-degree} 
Let the adjacency matrix $A$ have connectivity parameters $a,b$ such that $a+b > 2$. Then there exist $c_1,c_2,c_3 > 0$ such that $d_{\min} \geq c_1\log n$ and $d_{\max} \geq c_2\log n$ with probability at least $1-c_3(n^{1-(a+b)/2})$.
\end{lemma}

\begin{proof} 
Since $a+b \geq 2$, $A$ is connected with high probability. By Bennett's inequality, for any $u \in V$, $t < (a+b)/2$, and $h(x) = (1+x)\log(1+x)-x$ we have
\begin{align*}
&\Prob{d(u) \leq \frac{a+b}{2}\log n - t\log n} \\
&\qquad \qquad = \Prob{\sum_{v \in V} A(u,v) \leq \frac{a+b}{2}\log n - t\log n} \\
&\qquad \qquad\leq \exp\left(-\left(\sum_{v \in V} \Var{A(u,v)}\right)h\left(\frac{-t\log n}{\sum_{v \in V} \Var{A(u,v)}}\right)\right) \\
&\qquad \qquad\leq \exp\left(-\left(\sum_{v \in V} \Ex{A(u,v)^2}\right)h\left(\frac{-t\log n}{\sum_{v \in V} \Ex{A(u,v)^2}}\right)\right) \\
&\qquad \qquad = \exp\left(-\log n\left(\frac{a+b}{2}\right)h\left(\frac{-2t}{a+b}\right)\right)
\end{align*}

In addition, we know that $\left(\frac{a+b}{2}\right)h\left(\frac{-2t}{a+b}\right)$ is increasing with $t$ and 
\[
\lim_{t \rightarrow (a+b)/2}\left(\frac{a+b}{2}\right)h\left(\frac{-2t}{a+b}\right) = \frac{a+b}{2}
\]
Then for any $\eps > 0$, there exists a $t < (a+b)/2$ such that for some $c_3 > 0$,
\[
\Prob{d(v) \leq \frac{a+b}{2}\log n - t\log n} \leq c_3n^{-(a+b)/2 + \eps}.
\]  
Finally, taking a union bound over $O(n)$ unrevealed entries gives us
\[
\Prob{d_{\min} \geq \frac{a+b}{2}\log n - t\log n} \geq 1- c_3n^{-(a+b)/2 + 1 + \eps}. 
\]
We conclude by defining $c_1 = \frac{a+b}{2}-t$.

Applying the same chain of reasoning for the maximum, we note that
\[
\Prob{d(u) \geq \frac{a+b}{2}\log n + t_1\log n} \leq c_3n^{-(a+b)/2 + \eps}
\]  for some choice of $t_1 >0$, which gives us the same union bound 
\[
\Prob{d_{\max} \leq \frac{a+b}{2}\log n + t_1\log n} \geq 1- c_3n^{-(a+b)/2 + 1 + \eps}. 
\]
where $c_2 := (a+b)/2 + t_1$. \qedhere
\end{proof}

Note that the probability $1-c_3(n^{-(a+b)/2+1})$ is $1-o(n^{-1})$ for $a+b > 4$. Denote by $P_i^m$ the leave--one--out matrix given by \[
P_i^m(u,v) =\begin{cases}
      P_i(u,v), & \text{if}\ u,v \neq m \\
      \bar{P}_{i}(u,v), & \text{otherwise}.
    \end{cases}
\]

\begin{lemma}\label{lem:normalized-submatrix-norm} 
Assume that $n \max_{u,v} A(u,v) > c_0 \log n$ for some $c_0 > 1$. Then with probability $1-o(n^{-1})$, for any $r > 0$, there exist constants $C, c_2 > 0$ such that 
\[
\norm{P_i - \bar{P}_i}_2 < \frac{c_2}{\sqrt{\log n}}
\quad\text{and}\quad
\max_{m \in \V_i} \norm{P_i^{m} - \bar{P}_i}_2 <  \frac{c_2}{\sqrt{\log n}}.
\]
\end{lemma}

\begin{proof} 
We begin by noting that the first statement implies the second since for any $m \in \U$, $(P_i^{m} - \bar{P}_i)$ is a principal submatrix of  $(P_i - \bar{P}_i)$ and the $p,q$-norms for a principal submatrix $G \in \mathbb{R}^{k \times k}$ is bounded by that of the matrix $H \in \mathbb{R}^{n \times n}$, that is,
\[
\norm{G}_{q \rightarrow p} = \max_{x \in \mathbb{R}^{k}} \frac{\norm{G x}_{p}}{\norm{x}_{q}} = \max_{x \in \mathbb{R}^{k}} \frac{\norm{H (x, \mathbf{0}_{n-k})}_{p}}{\norm{x}_{q}} \leq \max_{y \in \mathbb{R}^{n}} \frac{\norm{H y}_{p}}{\norm{y}_{q}}  \leq \norm{H}_{q \rightarrow p}.
\]

Next, we see that 
\[
\norm{P_i - \bar{P}_i}_2 = \norm{D_i^{-1} A_i - \bar{D}_i^{-1}\bar{A}_i}_2  \leq \norm{D_i^{-1}(A_i-\bar{A}_i)}_2 + \norm{(\bar{D}_i^{-1}-D_i^{-1})\bar{A}_i}_2  \\
\]
The first term is bounded by Theorem \ref{th:lei-rinaldo} and the general fact that norms of principal submatrices are strictly smaller than norms of the original matrix.
\[
\norm{D_i^{-1}(A_i-\bar{A}_i)}_2 \leq \frac{\norm{(A_i-\bar{A}_i)}_2}{d_{\min}} \leq  \frac{\norm{(A-\bar{A})}_2}{d_{\min}} = \frac{c_2\sqrt{\log n}}{\log n} \leq \frac{c_2}{\sqrt{\log n}}
\]
with probability $1-o(n^{-1})$ given an appropriate choice of $c_2$ from Lemma \ref{lem:min-degree}.

We note that the degree $d_i(u) = d(u)$ when defined. For the second term, applying the same results tells us that with probability $1-o(n^{-1})$ for some $c_3$ we have
\begin{align*}
\norm{(\bar{D}_i^{-1}-D_i^{-1})\bar{A}_i}_2 &\leq \norm{(\bar{D}_i^{-1}-D_i^{-1})\bar{A}_i}_F \\
&= \sqrt{\sum_{u,v \in \V_i} \bar{A}_i(u,v)^2 \left(\frac{1}{d(u)} - \frac{1}{\bar{d}(v)}\right)^2}\\
&\leq \max_{u,v} \bar{A}(u,v) \sqrt{\sum_{u,v \in \V_i}\left(\frac{d(u)- \bar{d}(u)}{d(u)\bar{d}(u)}\right)^2} \\
&\leq \frac{\max_{u,v} \bar{A}(u,v)}{d_{\min}\bar{d}_{\min}} \sqrt{\sum_{u,v \in \V_i}(d(u)- \bar{d}(u))^2}\\
& \leq \frac{\sqrt{n}\max_{u,v} \bar{A}(u,v)}{d_{\min}\bar{d}_{\min}} \sqrt{\sum_{v \in \V_i}(d(v)- \bar{d}(v))^2}\\
&= \frac{\sqrt{n}\max_{u,v} \bar{A}(u,v)}{d_{\min}\bar{d}_{\min}} \norm{(A_i - \bar{A}_i)\Ind{}}_2 \\
&= \frac{\sqrt{n}\max_{u,v} \bar{A}(u,v)}{d_{\min}\bar{d}_{\min}} \norm{(A_i - \bar{A}_i)}_2 \cdot \norm{\Ind{}}_2 \\
& \leq \frac{C (n\max_{u,v} \bar{A}(u,v))^{\frac{3}{2}}}{\min(d_{\min},\bar{d}_{\min})^2} \\
&\leq \frac{c_3}{\sqrt{\log n}}. \qedhere
\end{align*} 

\end{proof}

%========================================
\subsection{Eigenvalue bounds}
%========================================
The approximation $\pi_i \approx P_i \bar{\pi}_i/\lambda_1$ requires bounding the norm of $P_i$ and the corresponding eigengap. Bounding the eigengap of the normalized submatrix first requires showing that the principal eigenvalue is bounded above and below by positive constants. 

\begin{lemma}\label{lem:eig-lbound}
When a $\delta$--fraction of nodes have been revealed, the expected substochastic matrix $\bar{P}_i$ has principal eigenvalue $\lambda_1$ such that $2/(2+\delta)\leq \lambda_1 \leq 1$.
\end{lemma}

\begin{proof}
The eigenvector $\pi_i$ is proportional to $x \Ind{\U_i} + y \Ind{\C_{-i}}$,  where $x,y$ are normalized such that $(1-\delta)x + y = 1$. Then solving for the eigenvalue explicitly, 
\begin{align*}
& (1-\delta)x\frac{a}{a+b} + y\frac{b}{a+b} = \lambda_1 x \\
& (1-\delta)x\frac{b}{a+b} + y\frac{a}{a+b} = \lambda_1 y 
\end{align*}
Summing the two equations and using the relationship of $x,y$ and $\delta$ gives us
\[
1 = (1-\delta)x + y = \lambda_1(x+y) =  \lambda_1(1+\delta x). 
\]
Therefore, since $0 < x < y$ and $0 < \delta < 1$ we get
\[
1 \geq \lambda_1= \frac{1}{(1+\delta x)} >  \frac{2}{(2+\delta)}.\qedhere 
\]
\end{proof}

As the normalized expectation is a rank-2 matrix, a bound on the eigengap follows naturally.

\begin{lemma}\label{lem:spectral-gap} 
For the matrix $\bar{P}_i$, the eigengap $\lambda_1 - \lambda_2 \geq  \delta^2/(2+\delta)$.
\end{lemma}

\begin{proof}  
The trace of the expected transition matrix $\bar{P}$ is $2a/(a+b)$. In addition, $ \lambda_1 + \lambda_2 = \tr(\bar{P}_i) = (a + (1-\delta)a)/(a+b)$. From Lemma \ref{lem:eig-lbound}, $\lambda_1 > 2/(2+\delta)$. Since $b \geq 0$, the eigengap is bounded by 
\[
\lambda_1 -\lambda_2 = 2\lambda_1 - \tr(\bar{P}_i) \geq  \frac{4}{2+\delta} - \frac{a + (1-\delta)a}{a+b} \geq \frac{4}{2+\delta} - (2-\delta) = \frac{\delta^2}{2+\delta}. \qedhere
\]
\end{proof}

We also consider a bound on the principal eigenvalue of $P_i$.

\begin{lemma}\label{lem:min-degree-perturbation}
There exists a constant $C > 0$ such that $|\lambda_1(P_i) - \lambda_1| \geq C/\sqrt{\log n}$ with probability $1-o(n^{-1})$.
\end{lemma}
\begin{proof}
By Weyl's inequality between eigenvalues and Lemmas \ref{lem:normalized-submatrix-norm} and \ref{lem:eig-lbound}, for some $c_1 > 0$ with probability $1-o(n^{-1})$ we have
\[
\lambda_1(P_i) \geq \lambda_1(\bar{P}_i) - \lambda_1(P_i-\bar{P}_i) \geq \lambda_1(\bar{P}_i) - \norm{P_i-\bar{P}_i}_2 \geq \lambda_1 - C/\sqrt{\log n}
\]
for sufficiently large $n$. The other direction follows analogously. \qedhere
\end{proof} 

\subsection{Bounding binomial differences}
%========================================
In this section, we bound the error rate of the mean-field approximation for the QSD and mixed methods by reducing bounds on the transition submatrices $P_i$ to bounds on the adjacency submatrices $A_i$. We note that using the approximations $A_i \bar{\pi}_i$ instead of $\mu_i$ and scaling by $\sqrt{n}$ gives the following expression an unrevealed node $u \in \U$
\begin{align}
\sigma(u)(\sqrt{n}\gamma(A_-\bar{\pi}_- -A_+\bar{\pi}_+)(u)+ \gamma_sS(u)) \label{eq:difference}
\end{align}
where $\gamma_s$ is a weight on the simple voting component. By Definitions \ref{def:qsd-def} and \ref{def:mixed-def}, we note that $\gamma_s = 0$ for the QSD method and that $\gamma_s = -1$ for the mixed method.

The following lemma bounds the error rate for each unrevealed node.

\begin{lemma}\label{lem:mean-field-error} 
For each unrevealed $u \in \U$, there exists $\eps > 0$  such that for any constant $\gamma_s$
\[
\Prob{\sigma(u)\left(\gamma\sqrt{n}(A_-\bar{\pi}_- -A_+\bar{\pi}_+)(u)+ \gamma_s S(u)\right) \leq \eps \log n}  \leq n^{-I(\rho-1,\rho+\gamma_s) + O(\eps)}
\]
where 
\[
I(\alpha,\beta) = \frac{1}{2}\sup_{\theta > 0} a+b - (1-\delta)ae^{-\theta \alpha} -(1-\delta)be^{\theta \alpha} -\delta ae^{-\theta \beta} - \delta be^{\theta \beta}.
\]
\end{lemma}
\begin{proof} We note by Definition \ref{def:pl-sbm} that $S(u) = \left(\sum_{v \in \R_+} A(u,v) - \sum_{v \in \R_-} A(u,v)\right)$ where $|\R_+| = |\R_-| = \delta n/2$. In addition, following Definition \ref{def:mf-eig}, $\gamma\sqrt{n} \bar{\pi}_i = \Ind{\U_i} + \rho \Ind{\C_{-i}}$. We show that for each $u \in \U$, the expression in \eqref{eq:difference} is equal in distribution to a difference of binomials. Rearranging the eigenvector difference \\ $\sigma(u)\left(\sqrt{n}\gamma(A_-\bar{\pi}_- -A_+\bar{\pi}_+)(u)\right)$ yields
\begin{align*}
&\sigma(u)\left(\sqrt{n}\gamma(A_-\bar{\pi}_- -A_+\bar{\pi}_+)(u)\right) \\
&\quad =\sigma(u)\left(\sum_{v \in \C_+} \rho A(u,v) + \sum_{v \in \U_-} A(u,v) - \sum_{v \in \C_-} \rho A(u,v) - \sum_{v \in \U_+} A(u,v)  \right) \\
&\quad =\sigma(u)\left((\rho-1)\left(\sum_{v \in \U_+} A(u,v)  - \sum_{v \in \U_-} A(u,v)\right)  +\rho\left(\sum_{v \in \R_+} A(u,v) - \sum_{v \in \R_-} A(u,v)\right) \right)
\end{align*}
Then we add in the $\sigma(u)\gamma_sS(u)$ term and show that the sum is equal in distribution to
\begin{align*}
&\sigma(u)\left((\rho-1)\left(\sum_{v \in \U_+} A(u,v)  - \sum_{v \in \U_-} A(u,v)\right)  +(\rho+\gamma_s)\left(\sum_{v \in \R_+} A(u,v) - \sum_{v \in \R_-} A(u,v)\right) \right)\\
&\quad \overset{d}{=} (\rho-1)\sum_{i=1}^{(1-\delta)n/2} (W_{n/2+1-i} - Z_{n/2+1-i}) + (\rho + \gamma_s)\sum_{i=1}^{\delta n/2}(W_i - Z_i)
\end{align*}
where the $W_i$, $Z_i$ are i.i.d.~Bernoulli$(a\log n/n)$ and Bernoulli$(b\log n/n)$ respectively. We apply Lemma \ref{lem:total-chernoff1} below with $\alpha = \rho-1$, $\beta = \rho+\gamma_s$ and choose $\eps(a,b) > 0$ such that
\[
\Prob{\sigma(u)(\gamma\sqrt{n}(A_-\bar{\pi}_- - A_+\bar{\pi}_+) + \gamma_s S)(u) \leq \eps\log n} \leq n^{-I(\rho -1,\rho+\gamma_s) + O(\eps)}. 
\]
for each $u \in \U$, where
\[
I(\alpha,\beta) = \frac{1}{2}\sup_{\theta > 0} a+b - (1-\delta)ae^{-\theta \alpha} -(1-\delta)be^{\theta \alpha} -\delta ae^{-\theta \beta} - \delta be^{\theta \beta}. \qedhere
\]
\end{proof}

We show the Chernoff bound for the difference of binomials in the following lemma.

\begin{lemma}\label{lem:total-chernoff1} 
Suppose $a > b$, the variables $W_i$ and $Z_i$ are i.i.d.~Bernoulli$(a\log n/n)$ and Bernoulli$(b\log n/n)$ and let $\delta \in (0,1)$. Then for any choice of $\gamma_s$, there exists an $\eps > 0$ independent of $n$ such that
\[
\Prob{\alpha \sum_{i=1}^{(1-\delta)n/2} (W_{n/2+1-i} - Z_{n/2+1-i}) + \beta \sum_{i=1}^{\delta n/2}(W_i - Z_i) \leq \eps \log n} \leq n^{-I(\alpha,\beta) +O(\eps)} 
\]
where $\rho$ is defined in Definition \ref{def:mf-eig} and
\[
I(\alpha,\beta) = \frac{1}{2}\sup_{\theta > 0} a+b - (1-\delta)ae^{-\theta \alpha} -(1-\delta)be^{\theta \alpha} -\delta ae^{-\theta \beta} - \delta be^{\theta \beta}.
\]
\end{lemma}

\begin{proof} 
We apply a Chernoff bound to see that for $\theta > 0$
\begin{align*}
&\Prob{\alpha\sum_{i=1}^{(1-\delta)n/2} (W_{n/2+1-i} - Z_{n/2+1-i}) + \beta \sum_{i=1}^{\delta n/2}(W_i - Z_i) \leq \eps \log n} \\
 &\qquad =\Prob{\alpha\sum_{i=1}^{(1-\delta)n/2} (Z_{n/2+1-i} - W_{n/2+1-i}) + \beta \sum_{i=1}^{\delta n/2}(Z_i - W_i) \geq -\eps \log n}\\
&\qquad\leq \Ex{e^{\theta(\alpha\sum_{i=1}^{(1-\delta)n/2}(Z_i - W_i)+\beta\sum_{i=1}^{\delta n/2}(Z_i - W_i))}} e^{\theta\eps\log n}\\
&\qquad\leq \Ex{e^{\theta\alpha\sum_{i=1}^{(1-\delta)n/2}(Z_{n/2+1-i} - W_{n/2+1-i})}}\Ex{e^{\theta\beta\sum_{i=1}^{\delta n/2}(Z_i - W_i)}} e^{\theta\eps\log n}\\
&\qquad= e^{\theta\eps\log n} \cdot \prod_{i=1}^{(1-\delta)n/2}\Ex{e^{\theta \alpha(Z_i-W_i)}}\cdot \prod_{i=1}^{\delta n/2} \Ex{e^{\theta \beta (Z_i - W_i)}}
\end{align*}
We note that \[
\log(\Ex{e^{k Z_i}}) = \log(e^{k} \cdot b\frac{\log n}{n} + 1-b\frac{\log n}{n}) \leq (e^{k}-1)\cdot b\frac{\log n}{n} 
\]
and that \[
\log(\Ex{e^{k W_i}}) = \log(e^{k} \cdot a\frac{\log n}{n} + 1-a\frac{\log n}{n}) \leq (e^{k}-1)\cdot a\frac{\log n}{n}. 
\] 
Then taking a log of the bound tells us that 
\begin{align*}
&\log \Prob{\alpha\sum_{i=1}^{(1-\delta)n/2}(Z_{n/2+1-i} - W_{n/2+1-i}) + \beta\sum_{i=1}^{\delta n/2}(Z_i - W_i) \geq -\eps \log n}  \\
&\qquad \leq \frac{\log n}{2}\left((1-\delta)a(e^{-\theta \alpha}-1)+(1-\delta)b(e^{\theta \alpha}-1)+ \delta a(e^{-\theta \beta}-1)+\delta b(e^{\theta \beta}-1) +2\theta\eps \right) \\
&\qquad \leq \frac{\log n}{2}\left((1-\delta)ae^{-\theta \alpha}+(1-\delta)be^{\theta \alpha}+ \delta ae^{-\theta \beta}+\delta be^{\theta \beta} -(a+b) + 2\theta \eps\right)
\end{align*}

Then substituting
\[
I(\alpha,\beta) = \frac{1}{2}\sup_{\theta > 0} a+b - (1-\delta)ae^{-\theta \alpha} -(1-\delta)be^{\theta \alpha} -\delta ae^{-\theta \beta} - \delta be^{\theta \beta}
\]
gives us the desired upper bound of 
\[
\Prob{\alpha\sum_{i=1}^{(1-\delta)n/2} (W_{n/2+1-i} - Z_{n/2+1-i}) + \beta \sum_{i=1}^{\delta n/2}(W_i - Z_i) \leq \eps \log n}  \leq n^{-I(\alpha,\beta)+O(\eps)}. \qedhere
\]

\end{proof}

\subsection{Entrywise bounds}
%========================================
In this section, we develop the entrywise bound that states $\norm{\pi_i -P_i\bar{\pi}_i/\lambda_1}_\infty = o(n^{-1})$ with high probability. We build off the techniques in Deng et al.~\cite{Deng2021} to prove this formally.

Following the leave--one--out technique from \cite{Abbe2020}, we first bound the difference between the eigenvectors of $P_i$ and $P_i^m$. Let $\pi_i^m$ be normalized so that $\norm{\pi_i^m}_2 = 1$ is the eigenvector of $P_i^m$ from the leave--one--out technique, where $m \in \U$ and
\[
P_i^m(u,v) =\begin{cases}
      P_i(u,v), & \text{if}\ u,v \neq m \\
      \bar{P}_{i}(u,v), & \text{otherwise}.
    \end{cases}
\]

As in Abb\'e et al.~\cite{Abbe2020}, we will use the Davis--Kahan theorem to bound the eigenvector perturbation by matrix perturbations for symmetric matrices. However, because the transition matrices $P_i = D^{-1}A$ are not symmetric, a generalization of Davis--Kahan is necessary. We use a simplified version of Theorem 3 in Deng et al.~\cite{Deng2021} for subspaces of dimension 1.

\begin{theorem}[Generalized Davis-Kahan]\label{th:Generalized-DK}
Let $M,N$ be symmetric matrices with $N$ positive definite, let $u$ be the principal eigenvector of $N^{-1}M$, let $\hat{u}$ be any vector, let $\hat{\lambda}$ be a constant, and let $\kappa(N)$ be the condition number of $N$. Then the angle $\theta = \arccos(\langle u, \hat{u}\rangle)$ between the two vectors satisfies
\[
\sin(\theta) \leq \frac{\kappa(N)||(N^{-1}M-\hat{\lambda}I)\hat{u}||_2}{|\hat{\lambda} -\lambda_2(N^{-1}M)|}.
\]
\end{theorem}

\begin{lemma}\label{lem:leave--one--out} 
For some $C > 0$, with probability $1-o(n^{-1})$ we have
\[
\max_{m \in \V_i} \norm{\pi_i - \pi_i^m}_2 \leq C\norm{\pi_i}_\infty.
\]
\end{lemma}

\begin{proof} 
Let $\theta$ be the angle between $\pi_i$ and $\pi_i^m$. By the law of cosines and the normalization $\norm{\pi_i}_2 = \norm{\pi_i^m}_2 = 1$, we get
\[
\norm{\pi_i - \pi_i^m}_2 = \sqrt{2-2\cos(\theta)} \leq \sqrt{2-2\cos^2(\theta)} = \sqrt{2} \sin(\theta).
\]
We use Theorem \ref{th:Generalized-DK} to bound $\sin(\theta)$. Setting $M = A_i^m, N = D_i^m, \hat{\lambda} = \lambda_1(P_i)$, and $\hat{u} = \pi_i$ we get
\[
\sin \theta  \leq \frac{\kappa(D_i^m)\norm{(P_i^m- P_i)\pi_i}_2 }{|\lambda_1(P_i)-\lambda_2(P_i^m)|}.
\]
We bound the denominator and numerator separately. First, by Lemma \ref{lem:min-degree}, $\kappa(D_i^m) \leq \frac{d_{\max}}{d_{\min}} \leq c_1$ for some $c_1 > 0$ for all $m \in \U$ with probability $1-o(n^{-1})$. In addition, by Weyl's inequality and Lemma \ref{lem:normalized-submatrix-norm} with probability $1-o(n^{-1})$ for some $c_2 > 0$
\[ 
\lambda_1(P_i)-\lambda_2(P_i^m) > \lambda_1 + \lambda_2-\norm{\bar{P}_i - P_i}_2 - \max_{m \in \V_i} \norm{\bar{P}_i - P_i^m}_2 \geq c_2 .
\]
It is sufficient to show that the numerator $(P_i^m- P_i)\pi_i = O\left(\norm{\pi_i}_\infty\right)$. To do so, we bound the $m$th entry of $(P_i^m- P_i)\pi_i$ and the other entries separately. For all entries $u \neq m$, we have
\[
\left|(P_i^m- P_i)\pi_i(u)\right| = \left|\left(\bar{P}_i(m,u)-P_i(m,u)\right)\pi_i(u)\right| \leq \left|\bar{P}_i(m,u)-P_i(m,u)\right|\norm{\pi_i}_\infty
\]
Otherwise, when $u = m$
\begin{align*}
\left|(P_i^m- P_i)\pi_i(m)\right| &= \left|\sum_{v \in \V_i} \left(\bar{P}_i(v,m)-P_i(v,m)\right)\pi_i(v)\right| \\
&\leq  \sum_{v \in \V_i} \left|\left(\bar{P}_i(v,m)-P_i(v,m)\right)\right|\cdot \norm{\pi_i}_\infty \\
&\leq  2\norm{\pi_i}_\infty
\end{align*}
Then we apply Lemma \ref{lem:normalized-submatrix-norm} to see that with probability $1-o(n^{-1})$ for some $c_3 > 0$
\begin{align*}
\norm{(P_i^m- P_i)\pi_i}_2^2 &\leq (2\norm{\pi_i}_\infty)^2 + \sum_{u \in \V_i, u \neq m}\left(\bar{P}_i(m,u)-P_i(m,u)\right)^2\norm{\pi_i}_\infty^2  \\
%&\leq  2\norm{\pi_i}_\infty + \norm{\pi_i}_\infty\left(\sum_{u \in \V_i, u \neq m}\left(\bar{P}_i(m,u)-P_i(m,u)\right)^2\right)^{\frac{1}{2}} \\
&\leq 4\norm{\pi_i}_\infty^2 + \norm{\bar{P}_i-P_i}_{2,\infty}^2 \norm{\pi_i}_\infty^2 \\
&\leq 4\norm{\pi_i}_\infty^2 + \norm{\bar{P}_i-P_i}_{2}^2 \norm{\pi_i}_\infty^2 \\
&\leq c_3\norm{\pi_i}_\infty^2
\end{align*}
We conclude that for some constant $C > 0$, with probability $1-o(n^{-1})$ 
\[
\max_{m \in \V_i} \norm{\pi_i - \pi_i^m}_2 \leq C\norm{\pi_i}_\infty. \qedhere
\] 
\end{proof}

\begin{lemma}\label{lem:mean-variance} 
There exists a $C > 0$ such that with probability $1-o(n^{-1})$
\[
\norm{\pi_i - \bar{\pi}_i}_2 \leq \frac{C}{\sqrt{\log n}}.
\]
\end{lemma}
\begin{proof}
Let $\theta$ be the angle between $\pi_i$ and $\bar{\pi}_i$. By the law of cosines, the normalization $\norm{\pi_i}_2 = \norm{\bar{\pi}_i}_2 = 1$ and Theorem \ref{th:Generalized-DK} with $N = D_i$, $M = A_i$, $\hat{u} = \bar{P}_i,$ and $\hat{\lambda} = \lambda_1(\bar{P}_i)$,
\[
\norm{\pi_i - \bar{\pi}_i}_2 = \sqrt{2 - 2\cos(\theta)} \leq \sqrt{2 - 2\cos^2(\theta)} = \sqrt{2}\sin(\theta) \leq \frac{\sqrt{2}\kappa(D_i)\norm{(P_i - \bar{P}_i)\bar{\pi}_i}_2}{\lambda_1 - \lambda_2(P_i)}.
\]
We note that \[\norm{(P_i - \bar{P}_i)\bar{\pi}_i}_2 \leq \norm{P_i - \bar{P}_i}_2\norm{\bar{\pi}_i}_2 \leq \norm{P_i - \bar{P}_i}_2\]
and that 
$\kappa(D_i) \leq \frac{d_{\max}}{d_{\min}} \leq c_1$ for some $c_1 > 0$ with probability $1-o(n^{-1})$.

By Theorem \ref{lem:normalized-submatrix-norm}, with probability $1-o(n^{-1})$
\[
\norm{P_i - \bar{P}_i}_2 =  \frac{c_2}{\sqrt{\log n}}.
\]
For the denominator, applying Weyl's inequality and Lemma \ref{lem:normalized-submatrix-norm} tells us that for some $c_3 > 0$ with probability $1-o(n^{-1})$ 
\[
\lambda_1 - \lambda_2(P_i)  \geq \lambda_1 - \lambda_2 - \norm{P_i - \bar{P}_i}_2 \geq c_3. \qedhere
\]
\end{proof}

We put together Lemmas \ref{lem:leave--one--out} and \ref{lem:mean-variance} to obtain an entrywise bound on the product $A_i(\pi_i -\bar{\pi}_i)$ below.

\begin{lemma}\label{lem:norm-eigenvector-prob}
There exists a $C > 0$ such that with probability $1-o(n^{-1})$
\begin{align*}
&\norm{A_i(\pi_i -\bar{\pi}_i)}_\infty \leq C \log n\left(\frac{\norm{\pi_i}_\infty}{\log \log n} + \frac{\log n }{\sqrt{n}\log \log n}\right). 
\end{align*}
\end{lemma}

\begin{proof} The proof closely follows the argument in Lemma 18 in \cite{Zhang2016} but we write it out here for completeness. We first bound $\norm{A_i(\pi_i -\bar{\pi}_i)}_\infty $.
We note that
\begin{align}
\norm{A_i(\pi_i -\bar{\pi}_i)}_\infty 
&= \max_{m \in \V_i} |\langle A_{i}(m,\,\cdot\,), \pi_i -\bar{\pi}_i\rangle| \notag \\
&\leq \max_{m \in \V_i} |\langle A_{i}(m,\,\cdot\,), \pi_i -\pi_i^m\rangle| + |\langle A_{i}(m,\,\cdot\,), \pi_i^m -\bar{\pi}_i\rangle| \notag \\
&\leq \max_{m \in \V_i} \norm{A_{i}}_{2,\infty}\norm{\pi_i -\pi_i^m}_2 \label{eq:term1}\\
&\quad + \max_{m \in \V_i} |\langle \bar{A}_{i}(m,\,\cdot\,), \pi_i^m -\bar{\pi}_i\rangle| \label{eq:term2}\\
&\quad + \max_{m \in \V_i} |\langle \bar{A}_{i}(m,\,\cdot\,)- A_{i}(m,\,\cdot\,), \pi_i^m -\bar{\pi}_i\rangle| \label{eq:term3}
\end{align}

For the first part of the sum \eqref{eq:term1}, by Theorem \ref{th:lei-rinaldo}, for some $c_1 > 0$
\[
\norm{A_{i}}_{2,\infty} \leq \norm{\bar{A}_i}_{2,\infty} + \norm{A_i - \bar{A}_i}_2 \leq c_1\sqrt{\log n}.
\] 
and by the leave--one--out construction in Lemma \ref{lem:leave--one--out}
\[
\max_{m \in \V_i} \norm{\pi_i -\pi_i^m}_2 \leq c_2\norm{\pi_i}_\infty
\]
both with probability $1-o(n^{-1})$.
Therefore,
\[
 \max_{m \in \V_i} \norm{A_{i}}_{2,\infty}\norm{\pi_i -\pi_i^m}_2  \leq c_1c_2 \norm{\pi_i}_\infty \sqrt{\log n}.
\]

For the second term \eqref{eq:term2}, for some $c_3 > 0$, with probability $1-o(n^{-1})$, we apply Lemma \ref{lem:mean-variance} to get
\begin{align*}
\max_{m \in \V_i} |\langle \bar{A}_{i}(m,\,\cdot\,), \pi_i^m -\bar{\pi}_i\rangle| &\leq \max_{m \in \V_i}  \norm{\bar{A}_i}_{2,\infty}\norm{\pi_i^m -\bar{\pi}_i}_2 \\
&\leq \norm{\bar{A}_i}_{2,\infty}\left(\max_{m \in \V_i} \norm{\pi_i - \pi_i^m}_2 + \norm{\pi_i - \bar{\pi}_i}_2\right) \\
&\leq \frac{c_3\log n}{\sqrt{n}}\left(\norm{\pi_i}_\infty  + \frac{1}{\sqrt{\log n}}\right) 
\end{align*}

Let $k = |\V_i| = \left(1-\frac{\delta}{2}\right)n$. The third term \eqref{eq:term3} uses the row-concentration property from~\cite{Deng2021}, that with probability at least $1-o(n^{-1})$
\[
\max_{m \in \V_i} |\langle(A_i(m,\,\cdot\,) - \bar{A}_i(m,\,\cdot\,)), (\pi_i^m - \bar{\pi}_i)\rangle| \leq c_4\left(\max_{m \in \V_i} \norm{w}_\infty \phi\left(\frac{\norm{w}_2}{\sqrt{k}\norm{w}_\infty}\right)\log n \right)
\]
where $w = \pi_i^m - \bar{\pi}_i$ and $\phi(t) = (1 \lor \log(1/t))^{-1}$ for $t > 0$.

Let $x=\sqrt{k}\norm{w}_\infty$, $y=\norm{w}_2$, then we have by Lemma \ref{lem:leave--one--out} with probability $1-o(n^{-1})$ that
\begin{align*}
\max_{m \in \V_i} x &= \sqrt{k}\max_{m \in \V_i} \norm{\pi_i^m - \bar{\pi}_i}_\infty \\
&\leq \sqrt{k}\left(\max_{m \in \V_i} \left(\norm{\pi_i^m - \pi_i}_2\right) + \norm{\pi_i}_\infty + \norm{\bar{\pi}_i}_\infty\right) \\
&\leq c_5\sqrt{n}\norm{\pi_i}_\infty.
\end{align*}
In addition, by Lemmas \ref{lem:leave--one--out} and \ref{lem:mean-variance}, with probability $1-o(n^{-1})$
\[
\max_{m \in \V_i} y = \max_{m \in \V_i} \norm{\pi_i^m - \bar{\pi}_i}_2 \leq \max_{m \in \V_i} \norm{\pi_i^m - \pi_i}_2 + \norm{\pi_i - \bar{\pi}_i}_2 \leq c_6\left(\norm{\pi_i}_\infty + \frac{1}{\sqrt{\log n}}\right).
\]

We see that for all possible values of $x,y$ we have
\begin{align*}
\norm{w}_\infty \phi\left(\frac{\norm{w}_2}{\sqrt{k}\norm{w}_\infty}\right)\log n  &\leq \frac{\log n}{\sqrt{n}}x\phi\left(\frac{y}{x}\right)\\
&= \frac{\log n}{\sqrt{n}}\left(\Ind{\frac{x}{y} \geq \sqrt{\log n}}x\phi\left(\frac{y}{x}\right) + \Ind{\frac{x}{y} < \sqrt{\log n}}y\frac{x}{y}\phi\left(\frac{y}{x}\right)\right)\\
&\leq \frac{\log n}{\sqrt{n}}\left(x\phi\left(\frac{1}{\sqrt{\log n}}\right) + y\sqrt{\log n}\phi\left(\frac{1}{\sqrt{\log n}}\right)\right).  
\end{align*}
as $\phi$ is monotone increasing while $\phi(t)/t$ is monotone decreasing.

As a result, for some constants $c_7, C > 0$, with probability $1-o(n^{-1})$ we have
\begin{align*}
&\max_{m \in \V_i} |\langle(A_i-\bar{A}_i)(m,\cdot), \pi_i^m - \bar{\pi}_i\rangle| \\
&\qquad \qquad \leq c_4\max_{m \in \V_i} \frac{\log n}{\sqrt{n}}\left(x\phi\left(\frac{1}{\sqrt{\log n}}\right) + y\sqrt{\log n}\phi\left(\frac{1}{\sqrt{\log n}}\right)\right)\\
&\qquad \qquad \leq \frac{c_4\log n \phi\left(\frac{1}{\sqrt{\log n}}\right)}{\sqrt{n}}\left(c_5\sqrt{n}\norm{\pi_i}_\infty  + c_6\sqrt{\log n}\norm{\pi_i}_\infty + c_6\right)\\
&\qquad \qquad \leq\frac{c_7\log n}{\sqrt{n}\log\log n}\left(\sqrt{n}\norm{\pi_i}_\infty + 1 \right)\\
&\qquad \qquad \leq C\left( \frac{\norm{\pi_i}_\infty\log n }{\log \log n} + \frac{\log n }{\sqrt{n}\log \log n}\right).
\end{align*}
as desired. We conclude that $\norm{A_i(\pi_i -\bar{\pi}_i)}_\infty \leq C \log n\left(\dfrac{\norm{\pi_i}_\infty}{\log \log n} + \dfrac{\log n }{\sqrt{n}\log \log n}\right)$.
\end{proof}

We use the previous lemma to prove an explicit entrywise bound for $P_i(\pi_i -\bar{\pi}_i)$.

\begin{lemma}\label{lem:norm-eigenvector} 
There exists a $C > 0$ such that with probability $1-o(n^{-1})$
\[
\norm{P_i(\pi_i -\bar{\pi}_i)}_\infty \leq \frac{C}{\sqrt{n} \log \log n}.
\]
\end{lemma}

\begin{proof} 
The result follows from Lemma \ref{lem:norm-eigenvector-prob} if we show that for some $c_1 > 0$, with probability $1-o(n^{-1})$
\[
\norm{\pi_i}_\infty  \leq \frac{c_1}{\sqrt{n}}.
\]

We begin by expanding
\[
\norm{\pi_i}_\infty = \norm{\frac{1}{\lambda_1(P_i)} P_i \pi_i}_\infty \leq \norm{\frac{1}{\lambda_1(P_i)} P_i \bar{\pi}_i}_\infty + \norm{\frac{1}{\lambda_1(P_i)} P_i(\pi_i-\bar{\pi}_i)}_\infty.
\]

We bound the second term with Lemma \ref{lem:norm-eigenvector-prob} by pulling out the degree matrix $D^{-1}$ and we  account for dividing by the eigenvalue $\lambda_1(P_i)$ by applying our bound in Lemma \ref{lem:min-degree-perturbation}. We then see that with probability $1-o(n^{-1})$, for some $c_2 > 0$
\[
\lambda_1(P_i) \geq c_2,
\]
giving upper bounds with high probability of
\begin{align*}
&\norm{\frac{1}{\lambda_1(P_i)} P_i \bar{\pi}_i}_\infty  \leq c_2 \norm{P_i}_\infty \norm{\bar{\pi}_1}_\infty \leq \frac{c_3}{\sqrt{n}}. \\
&\norm{\frac{1}{\lambda_1(P_i)} P_i(\pi_i-\bar{\pi}_i) }_\infty \leq \norm{D^{-1}}_{\infty}\norm{\frac{1}{\lambda_1(P_i)} A_i(\pi_i-\bar{\pi}_i) }_\infty  \leq \frac{c_4\norm{\pi_i}_\infty}{\log \log n} + \frac{c_4}{\sqrt{n}\log \log n}.
\end{align*}
We conclude that $\norm{\pi_i}_\infty  \leq c_1/\sqrt{n}$ for some $c_1 > 0$ with probability $1-o(n^{-1})$. \qedhere
\end{proof}

We conclude our entrywise analysis by giving an explicit order to the following expression.

\begin{lemma}\label{lem:perturbed-eigenvector}  
There exists a $C > 0$ such that with probability $1-o(n^{-1})$
\[
\norm{\pi - \frac{P_i\bar{\pi}_i}{\lambda_1}}_\infty \leq \frac{C}{\sqrt{n}\log\log n}.
\]
\end{lemma}
\begin{proof} 
The result follows from expanding $\pi_i$ as an eigenvector and applying Lemma \ref{lem:norm-eigenvector}. In addition, we also use the bound on $\norm{\pi_i}_\infty = O(n^{-1/2})$ in Lemma \ref{lem:norm-eigenvector}, as well as Lemma \ref{lem:min-degree-perturbation} to show that there exist $c_1,c_2 > 0$ such that with probability $1-o(n^{-1})$
\begin{align*}
\norm{\pi - \frac{P_i\bar{\pi}_i}{\lambda_1}}_\infty &\leq \frac{1}{\lambda_1}\norm{P_i(\pi_i-\bar{\pi}_i)}_\infty + \norm{P_i}_\infty \norm{\pi_i}_\infty \left(\frac{1}{\lambda_1(P_i)}- \frac{1}{\lambda_1}\right) \\
&\leq  \frac{c_1}{\sqrt{n}\log\log n} + \frac{c_2}{\sqrt{n\log n}}.
\end{align*}
The result follows.
\end{proof}

\subsection{Equivalence of left and right transition eigenvectors}

\begin{lemma} \label{lem:left-right-eigs} 
For connected graphs, the left eigenvector $\mu_i$ of the transition matrix $P_i$ satisfies $\mu_i \propto D_i\pi_i$.
\end{lemma}
\begin{proof} It is sufficient to show that $D_i\pi_i$ is a left eigenvector of $P_i$ with the eigenvalue $\lambda_1(P_i)$. Using that $\pi_i$ is a principal right eigenvector of $P_i$ we have
\[
P_i^T D_i\pi_i 
= A_iD_i^{-1}D_i\pi_i 
= D_iD_i^{-1}A_i\pi_i 
= D_iP_i\pi_i 
= \lambda(P_i)D_i\pi_i \qedhere
\]
\end{proof}

We use Lemma \ref{lem:left-right-eigs} to bound the difference between using the left and right eigenvectors for classification.

\begin{lemma}\label{lem:left-eig-difference} 
Let $\bar{c} = \norm{\bar{A}_i\bar{\pi}_i}_1/\sqrt{n}\log n$. Then there exists a $C > 0$ such that with probability $1-o(n^{-1})$ we have
\[
\norm{(\mu_- -\mu_+) - \frac{\lambda_1}{\bar{c}\sqrt{n}\log n}(D_-\pi_--D_+\pi_+)}_\infty \leq \frac{C}{n\log\log n}
\]
where the differences are defined by restricting to unseen nodes $u \in \U$.    
\end{lemma}

\begin{proof} 
By Lemma \ref{lem:left-right-eigs} we have $\mu_i = D_i\pi_i/\norm{D_i\pi_i}_1$. Let us first denote the normalizing factor for the right eigenvectors by $\tilde{c} = \bar{c}\sqrt{n}\log n/\lambda_1$.
By applying  Lemma \ref{lem:left-right-eigs}, we see that

\begin{align*}
&\norm{(\mu_- -\mu_+) - \frac{1}{\tilde{c}}(D_-\pi_--D_+\pi_+)}_\infty \\
&\qquad\qquad=  \norm{\frac{1}{\norm{D_-\pi_-}_1}D_-\pi_- -\frac{1}{\norm{D_-\pi_-}_1}D_+\pi_+ -\frac{1}{\tilde{c}}(D_-\pi_--D_+\pi_+)}_\infty \\
&\qquad\qquad\leq \norm{\left(\frac{1}{\tilde{c}}-\frac{1}{\norm{D_-\pi_-}_1}\right)D_-\pi_-}_\infty + \norm{\left(\frac{1}{\tilde{c}}-\frac{1}{\norm{D_+\pi_+}_1}\right)D_+\pi_+}_\infty 
\end{align*}

Both terms are identically distributed by exchangeability, so it is sufficient to bound
\[
\norm{\left(\frac{1}{\tilde{c}}-\frac{1}{\norm{D_i\pi_i}_1}\right)D_i\pi_i}_\infty \leq \frac{1}{\norm{D_i\pi_i}_1 \tilde{c}}\left|\norm{D_i\pi_i}_1- \tilde{c}\right|\norm{D_i}_\infty\norm{\pi_i}_\infty
\]
for either community $i \in \{+,-\}$. We proceed term by term. First, we note by Lemma \ref{lem:norm-eigenvector} and Lemma \ref{lem:min-degree} that there exist $c_1,c_2 > 0$ such that with probability $1-o(n^{-1})$, $\norm{\pi_i}_\infty \leq c_1/\sqrt{n}$ and $\norm{D_i}_\infty = d_{\max} \leq c_2\log n$. Next, we have that $\tilde{c} = \norm{\tilde{A}_i\tilde{\pi}_i}_1/\lambda_1 = \Theta\left(1/\sqrt{n}\log n\right)$. In addition, since $\norm{\pi_i}_\infty = O\left(1/\sqrt{n}\right)$, it holds for some $c_3 > 0$ with probability $1-o(n^{-1})$, $(\norm{D_i\pi_i}_1)^{-1} \leq c_3/(\sqrt{n}\log n)$. 
Finally, we bound $\left|\norm{D_i\pi_i}_1- \tilde{c}\right|$ by the reverse triangle inequality, Lemma \ref{lem:perturbed-eigenvector}, and Theorem \ref{th:lei-rinaldo}. As a result, there exist $c_4,c_5,c_6 > 0$ such that with probability $1-o(n^{-1})$ we have
\begin{align*}
\left|\norm{D_i\pi_i}_1- \tilde{c}\right| &\leq \norm{D_i\pi_i - \frac{\bar{A}_i\bar{\pi}_i}{\lambda_1}}_1\\
&\leq \norm{D_i\pi_i- \frac{D_iP_i\bar{\pi}_i}{\lambda_1}}_1 +\norm{\frac{D_iP_i\bar{\pi}_i}{\lambda_1} - \frac{\bar{A}_i\bar{\pi}_i}{\lambda_1}}_1 \\
&\leq \norm{D_i}_\infty\norm{\pi_i- \frac{P_i\bar{\pi}_i}{\lambda_1}}_1 +\frac{1}{\lambda_1}\norm{A_i\bar{\pi}_i - \bar{A}_i\bar{\pi}_i}_1 \\
&\leq n\norm{D_i}_\infty\norm{\pi_i- \frac{P_i\bar{\pi}_i}{\lambda_1}}_\infty +\frac{\sqrt{n}}{\lambda_1}\norm{A_i\bar{\pi}_i - \bar{A}_i\bar{\pi}_i}_2 \\
&\leq c_4\left(n\log n\norm{\pi_i- \frac{P_i\bar{\pi}_i}{\lambda_1}}_\infty +\sqrt{n}\norm{A_i - \bar{A}_i}_2\right) \\
&\leq c_5\left(\frac{n\log n}{\sqrt{n}\log \log n} +\sqrt{n\log n}\right) \\
&\leq c_6 \frac{\sqrt{n}\log n}{\log \log n}
\end{align*}
Combining the bounds, there exists $C > 0$ such that with probability $1-o(n^{-1})$ we get
\[
\norm{\left(\frac{1}{\tilde{c}}-\frac{1}{\norm{D_i\pi_i}_1}\right)D_i\pi_i}_\infty \leq \frac{C}{2}\left(\frac{1}{n\log^2n} \cdot \frac{\sqrt n \log n}{\log \log n} \cdot \log n \cdot \frac{1}{\sqrt n}\right) = \frac{C}{2n\log \log n}
\]
The result follows.
\end{proof}

\subsection{Entrywise error rate for right eigenvectors}
The entrywise error rate of $Q(u)$ is $\Prob{\sigma(u)Q(u) < 0}$, which is bounded above by $\Prob{\sigma(u)Q(u) \leq \eps\log n}$, where we include a noise term of $\eps \log n$ for some $\eps > 0$. From Lemma \ref{lem:left-eig-difference}, we note that we can approximate $Q(u) = (\mu_- - \mu_+)(u)$ by $d(u)(\pi_- - \pi_+)(u)$ up to a constant factor. To analyze both the QSD and mixed methods, we include a weight $\gamma_s$ for a simple voting component $S(u)$ and rescale the quasi-stationary component appropriately. By Definition \ref{def:mf-eig}, the mean-field eigenvector $\bar{\pi}_i$ satisfies $\gamma\sqrt{n}\bar{\pi}_i = \Ind{\U_i} + \rho \Ind{\C_{-i}}$. An additional factor of $\lambda_1$ follows from the approximation $\pi_i \sim P_i \bar{\pi}_i/\lambda_1$ inspired by the approach in \cite{Abbe2020}. Scaling by these constants gives the following lemma.

\begin{lemma}\label{lem:right-eig-error}
For any unseen node $u \in \U$ and any constant $\gamma_s$, we have that 
\[
\Prob{\sigma(u)\left(\gamma\lambda_1\sqrt{n}d(u)(\pi_--\pi_+)(u) + \gamma_s S(u)\right)\leq \eps \log n} \leq n^{-I(\rho-1,\rho + \gamma_s) + O(\eps)} +o(n^{-1})
\]
where $\gamma$ is defined in Definition \ref{def:mf-eig} and
\[
I(\alpha,\beta) = \frac{1}{2}\sup_{\theta > 0} a+b - (1-\delta)ae^{-\theta \alpha} -(1-\delta)be^{\theta \alpha} -\delta ae^{-\theta \beta} - \delta be^{\theta \beta}.
\]
\end{lemma}
\begin{proof}
We first expand out each eigenvector $\pi_i$ in terms of the mean-field eigenvector $\bar{\pi}_i$ writing
\[
\pi_i = \frac{P_i \bar{\pi}_i}{\lambda_1} + \left(\pi_i - \frac{P_i \bar{\pi}_i}{\lambda_1}\right)
\]
and get an entrywise difference of 
\begin{equation*}
(\pi_- - \pi_+)(u) = \frac{1}{\lambda_1}(P_-\bar{\pi}_-(u) -P_+\bar{\pi}_+ (u)) +  \left(\pi_- - \frac{P_-\bar{\pi}_-}{\lambda_1}\right)(u) - \left(\pi_+ - \frac{P_+\bar{\pi}_+}{\lambda_1}\right)(u)
\end{equation*}
for any $u \in \U$. Then we expand the eigenvectors to get
\begin{align}
    &\Prob{\sigma(u)\left(\gamma\lambda_1\sqrt{n}d(u)(\pi_--\pi_+)(u)+\gamma_s S(u)\right) \leq \eps \log n} \notag \\
    &\leq \Prob{\sigma(u)\left(\gamma\lambda_1\sqrt{n}d(u)\left(\frac{P_-\bar{\pi}_-}{\lambda_1}-\frac{P_+\bar{\pi}_+}{\lambda_1}\right)(u) +\gamma_sS(u)\right)\leq 2\eps \log n} \label{eq:main-left}  \\
    &\quad +\Prob{\sigma(u)\gamma\lambda_1\sqrt{n}d(u)\left(\left(\pi_- - \frac{P_-\bar{\pi}_-}{\lambda_1}\right) - \left(\pi_+ - \frac{P_+\bar{\pi}_+}{\lambda_1}\right)\right) (u) \leq -\eps \log n} \label{eq:main-right} 
\end{align}

Next, we note that $D_i P_i = A_i$, allowing us to simplify  \eqref{eq:main-left} to a scaled difference of binomials. We then apply the Chernoff bound in Lemma \ref{lem:mean-field-error} to show that there exists an $\eps > 0$ independent of $n$ such that for each $u \in \U$, we bound \eqref{eq:main-left} by
\begin{align*}
&\Prob{\sigma(u)\left(\gamma\lambda_1\sqrt{n}d(u)\left(\frac{P_-\bar{\pi}_-}{\lambda_1}-\frac{P_+\bar{\pi}_+}{\lambda_1}\right)(u) +\gamma_s S(u)\right)\leq 2\eps \log n} \\
&\qquad= \Prob{\sigma(u)\left(\gamma\sqrt{n}(A_-\bar{\pi}_--A_+\bar{\pi}_+)(u) +\gamma_s S(u) \right)\leq 2\eps\log n} \\
&\qquad \leq  n^{-I(\rho-1,\rho+\gamma_s) + O(\eps)}
\end{align*}
where $I$ is defined as in Lemma \ref{lem:mean-field-error} and
\[
I(\rho-1,\rho+\gamma_s) = \frac{1}{2}\sup_{\theta > 0} (a+b) - (1-\delta)ae^{-\theta(\rho-1)} -(1-\delta)be^{\theta(\rho-1)} -\delta ae^{-\theta (\rho+\gamma_s)} - \delta be^{\theta (\rho+\gamma_s)}.
\]

Finally, by Lemma \ref{lem:perturbed-eigenvector}, $\norm{\pi_i - \frac{P_i\bar{\pi}_i}{\lambda_1}}_\infty  = o\left(n^{-1/2}\right)$ with probability $1-o(n^{-1})$. When this holds, via the triangle inequality we have
\[
\norm{\left(\pi_- - \frac{P_-\bar{\pi}_-}{\lambda_1}\right) - \left(\pi_+ - \frac{P_+\bar{\pi}_+}{\lambda_1}\right)}_\infty = o\left(n^{-1/2}\right).
\] 
We relate this to \eqref{eq:main-right} via a bound on the minimum degree $d_{\min} = \Theta(\log n)$ in Lemma \ref{lem:min-degree} that holds with probability $1-o(n^{-1})$. In addition, by Lemma \ref{lem:eig-lbound}, $\lambda_1$ is of constant order. Therefore, setting $\eps_1 = (\eps \log n)/(\lambda_1 d_{\min})  = \Theta(1)$ and taking any unseen node $u \in \U$, \eqref{eq:main-right} satisfies 
\begin{align*}
&\Prob{\sigma(u)\lambda_1\sqrt{n}d(u)\left(\left(\pi_- - \frac{P_-\bar{\pi}_-}{\lambda_1}\right) - \left(\pi_+ - \frac{P_+\bar{\pi}_+}{\lambda_1}\right)\right) (u) \leq -\eps \log n} \\
&\qquad \leq \Prob{\sigma(u)\left(\left(\pi_- - \frac{P_-\bar{\pi}_-}{\lambda_1}\right) - \left(\pi_+ - \frac{P_+\bar{\pi}_+}{\lambda_1}\right)\right) (u) \leq -\frac{\eps_1}{\sqrt{n}}} \\
& \qquad =  o(n^{-1})
\end{align*}
for any constant $\eps_1 > 0$.
Then we combine our bounds on \eqref{eq:main-left} and \eqref{eq:main-right} to get
\[
\Prob{\sigma(u)\left(\gamma\lambda_1\sqrt{n}d(u)(\pi_--\pi_+)(u) + \gamma_s S(u)\right)\leq \eps \log n} \leq n^{-I(\rho-1,\rho+\gamma_s) + O(\eps)} + o(n^{-1}). \qedhere
\]    
\end{proof}

\section{Experiments and Discussion}\label{sec:experiments}
%========================================

In this section, we compare the proposed methods over both simulated and real-world datasets. In all cases, we reveal approximately 10\% of nodes. The simulated data follows a connected and bounded degree SBM with connectivity parameters $a=4$ and $b=1$, which are below the exact recovery rate. The real datasets come from Yang et al.~\cite{Yang12}. In the analysis, we consider all graphs as undirected. In Figure \ref{fig:qsd} we plot the QSD eigenvector entries $\left(\mu_-(u),\mu_+(u)\right)$ for the connected SBM and their difference $\mu_-(u) - \mu_+(u)$.

In general, more information should not hurt performance. However, we note that applying $k$-means on the two quasi-stationary eigenvectors leads to a linearly separable problem in the plane for the connected SBM case with parameters as mentioned above. Panel \ref{fig:qsd-2d} plots the two quasi-stationary eigenvectors $(\mu_+,\mu_-)$ with the communities labeled. An alternative approach in two dimensions is classification by support vector machine, which runs into the issue that the labeled nodes are only defined along one coordinate. Another alternative is spectral clustering via similarity graphs, as implemented by von Luxburg~\cite{Luxburg07}. We will further explore appropriate embeddings of the labeled nodes in the plane in future work.

\begin{figure*}[ht]
    \centering
    \begin{subfigure}[t]{0.49\linewidth}
        \centering
        \includegraphics[width=1\linewidth]{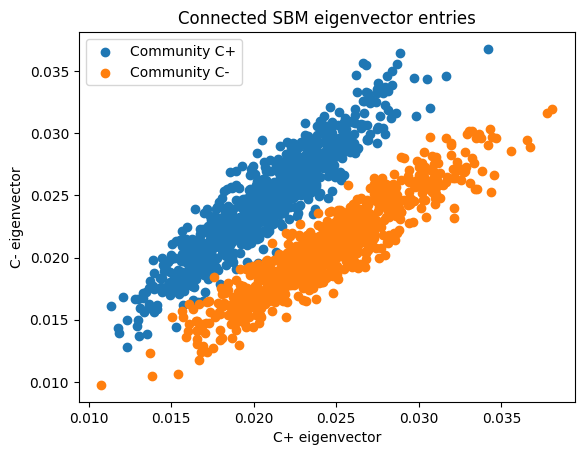}
        \caption{QSD eigenvector entries.}
        \label{fig:qsd-2d}
    \end{subfigure}
    ~
    \begin{subfigure}[t]{0.49\linewidth}
        \centering
        \includegraphics[width=1\linewidth]{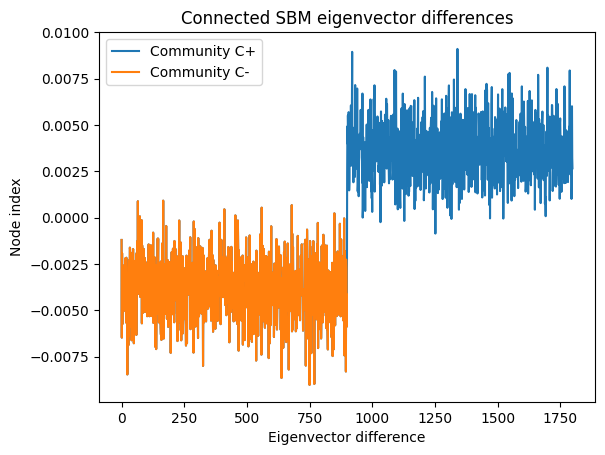}
        \caption{QSD eigenvector difference.}
        \label{fig:qsd-diff}
    \end{subfigure} 
    \caption{QSD eigenvectors on connected SBM.}\label{fig:qsd}
\end{figure*}

In Table \ref{tab:qsd_experiments}, we show the number of nodes for each graph, the recovery rate from the unlabeled nodes when running the standard spectral clustering and the quasi-stationary clustering. We computed the k-means based on the QSD eigenvectors but omitted them due to subpar performance. Since isolated nodes cannot be consistently classified and random walks cannot cross disconnected components, we restrict all analyses to the giant component of the respective datasets.

\begin{table}[ht]
    \centering
    \begin{tabular}{|c|c|c|c|c|}
        \hline
        Dataset & Nodes & Spectral Recovery Rate & QSD Recovery Rate\\
        \hline
        SBM-connected & 2000 & \textbf{1} & 0.992 \\
        SBM-bounded-degree & 2000 & .502 & \textbf{0.81}\\
        IONOSPHERE & 351 & \textbf{.69} & 0.67\\
        DIABETES & 768 & .522 & \textbf{0.67} \\
        WDBC & 683 & \textbf{.70} & 0.55\\
        POLBLOGS & 1224 & .915 & \textbf{0.943}\\
        SPAM & 4601 & .673 & \textbf{0.70 }\\
        GISETTE & 7000 & .906 & \textbf{0.959}\\
        \hline
    \end{tabular}
    \caption{Comparison of Recovery Rates for the Standard Spectral and QSD methods with a $\delta = 0.1$ fraction of revealed nodes.}
     \label{tab:qsd_experiments}
\end{table}

We see that especially in the sparse case, side information allow us to dramatically outperform the standard spectral recovery rate. 

In this work, we examined quasi-stationary algorithms for community detection on Stochastic Block Model (SBM) with two balanced communities and partially revealed labels under noiseless conditions. We introduced a class of single-step estimators based on quasi-stationary distributions, which combine a spectral component and a simple voting scheme utilizing revealed nodes. By aligning a random walk intuition with a spectral algorithm, we achieved asymptotically optimal recovery rates and a useful way to combine side information from partially revealed communities. This was accomplished by extending entrywise eigenvector analysis from the adjacency and Laplacian matrices to the transition submatrices corresponding to revealed nodes within each community.

In the connected regime, tight information-theoretic limits imply that nearly all labels must be revealed to change the exact recovery threshold. Our bound in Theorem \ref{th:main-result} for the mixed method gives the optimal asymptotic rate by a careful tuning of the weight to the quasi-stationary component. Future parameter tuning may yield non-asymptotic improvements in the error rate. In the bounded degree regime, this framework may yield additional insight for the statistical physics conjectures of Zhang, Moore, and Zdeborov\'a.~\cite{Zhang14} on partial recovery with side information, and our experimental findings suggest significant potential gains from side information.

In future extensions of the PL--SBM model and analysis, we will consider random graphs with weighted edges, extend the approach to networks with multiple communities, and derive error bounds in the bounded degree regime.

\bibliographystyle{abbrvnat}
\bibliography{references}
\end{document}